\documentclass[12pt,amstex]{amsart}

\usepackage{mathptmx}%字体
\usepackage{mathrsfs}

\usepackage{verbatim}
\usepackage{url}
\usepackage[all]{xy}
\usepackage{color}
\usepackage[colorlinks=true,citecolor=blue]{hyperref}% 引用和定理颜色

\usepackage{stmaryrd}
\usepackage{epsfig}
\usepackage{amsmath}
\usepackage{amssymb}

\usepackage{amscd}
\usepackage{graphicx}
\usepackage{pstricks}
\usepackage{tocvsec2}

\theoremstyle{plain}
\newtheorem{thm}{Theorem}[section]
\newtheorem{lem}[thm]{Lemma}

\newtheorem{cor}[thm]{Corollary}
\theoremstyle{definition}
\newtheorem{de}[thm]{Definition}

\newtheorem{rem}[thm]{Remark}
\numberwithin{equation}{section}
\newtheorem{claim}{Claim}

\def \N {\mathbb N}

\def \Z {\mathbb Z}
\def \R {\mathbb R}

\def \F {\mathcal F}

\def \SS {\mathcal{S}}

\def \a {\alpha }
\def \b {\beta}
\def \ep {\epsilon}

\def \D {\Delta}

\def \lel {\left\lceil}
\def \rr {\right\rceil}

\makeatletter
\@namedef{subjclassname@2020}{\textup{2020} Mathematics Subject Classification}
\makeatother

\begin{document}

\title[Topologically mildly mixing along generalized polynomials]{Topologically mildly mixing of higher orders along generalized polynomials}

\author[Y.~Cao]{Yang Cao}
\address[Yang Cao]{Department of Mathematics, Nanjing University, Nanjing 210093, P.R. China}
\email{cy412@nju.edu.cn}

\author[J.~Zhao]{Jianjie Zhao$^*$}\let\thefootnote\relax\footnote{* Corresponding author.}
\address[Jianjie Zhao]{School of Mathematics, Hangzhou Normal University, Hangzhou 311121, P.R. China}
\email{zjianjie@hznu.edu.cn}

\subjclass[2020]{Primary: 37B20, Secondary: 37B05, 11B83, 37A25}

\keywords{generalized polynomials; mildly mixing; recurrence}

%\thanks{This research is supported by NNSF of China (11971455, 11571335).}

%\date{2023.09.22}

\begin{abstract}
	This paper is devoted to studying the multiple recurrent property of topologically mildly mixing systems along generalized polynomials. We show that if a minimal system is topologically mildly mixing, then it is mild mixing of higher orders along generalized polynomials. Precisely, suppose that $(X, T)$ is a topologically mildly mixing minimal system, $d\in \mathbb{N}$, $p_1, \dots, p_d$ are integer-valued generalized polynomials with $(p_1, \dots, p_d)$  non-degenerate. Then for all non-empty open subsets $U , V_1, \dots, V_d $ of $X$, $$\{n\in \Z: U\cap T^{-p_1(n) }V_1 \cap \dots \cap T^{-p_d(n) }V_d \neq \emptyset \}$$ is an IP$^*$-set.

\end{abstract}

\maketitle

%\markboth{ergodic}{S. Shao and X.D. Ye}

%\newpage

%\tableofcontents \settocdepth{subsection}

%\newpage

\section{Introduction}

Throughout this paper a \emph{topological dynamical system} (or \emph{dynamical system}, \emph{system} for short) is a pair $(X, T),$ where $X$ is a non-empty compact metric space with a metric $\rho$ and $T$  is a homeomorphism from $X$ to itself.  
%and by a measure preserving system, we mean a quadruple (X, B, ?, T ), where (X, B) is a standard Borel space, ? is a probability measure on (X, B) and T : (X, B, ?) ? (X, B, ?) is an invertible measure preserving transformation.

Mixing is a fundamental notion in qualitative theory of dynamical systems, it refers to indecomposability of dynamics and synchronization of transfer times between open sets. In the hierarchy of mixing there is a notion, namely that of mildly mixing which was introduced by Glasner and Weiss \cite{GW06}, and independently by Huang and Ye \cite{HY2004}. It is well-known that  mild mixing is stronger than weak mixing, and weaker than strong mixing. A dynamical system $(X,T)$ is said to be \emph{topologically mildly mixing} if the product system with any transitive system is still transitive. As is shown in \cite{HY2004}, when $(X,T)$ is minimal, another characterization of mild mixing, via the hitting time set of two non-empty open subsets, goes as follows: a minimal system $(X,T)$ is topologically mildly mixing if for all non-empty open subsets $U,V$ of $X,$ $$N(U,V)=\{n\in \Z: U\cap T^{-n}V\neq \emptyset\}$$ is an IP$^*$-set. (A subset $A$ of an additive semigroup $S$ is \emph{IP$^*$} if for any sequence $\{n_i\}$ in $S$, there is some finite, non-empty set $\a\subset \N$ such that $\sum_{i\in \a}n_i\in A$.)   %(see \cite{HY2004} for details). 

In light of this characterization, it is natural to ask whether one may obtain, e.g., $$N(U,V_1,V_2)=\{n\in \Z: U\cap T^{-n}V_1 \cap T^{-2n}V_2\neq \emptyset\}$$ is an IP$^*$-set. The answer is yes. In particular, Cao and Shao has shown in \cite{CS22} that in  mildly mixing minimal systems, for any $d\in \N$ and non-empty open subsets $U, V_1, \dots, V_d$ of $X,$
\begin{equation}
	\{n\in \Z: U\cap T^{-n}V_1 \cap\dots \cap T^{-dn}V_d\neq \emptyset\}
\end{equation}
is an IP$^*$-set, which is entitled ``Mildly mixing of all orders.'' However, linear orders proved not to be ``all''. Indeed, as was shown in \cite{CS22}, for mildly mixing minimal systems, suppose $p_i(x)\in \Z[x],$ $1\le i\le d,$ having the property that no $p_i$ and no $p_i-p_j$ is constant, $1\le i\neq j\le d,$ one obtains 
\begin{equation}\label{ordinary polynomials}
\{n\in \Z: U\cap T^{-p_1(n)}V_1 \cap\dots \cap T^{-p_d(n)}V_d\neq \emptyset\}
\end{equation} 
is an IP$^*$-set.
Thus we see that mildly  mixing implies not merely mildly mixing of higher linear orders but of higher polynomial orders as well. The next question that arises is this: do such polynomial functions constitute a suitably ``most general class'' of integer sequences along which mildly mixing minimal systems are well behaved? In this paper, we give a negative answer.

The goal of this paper is extending (\ref{ordinary polynomials}) to classes of sequences which are more general than those of the form $p(n),$ where $p$ is a polynomial satisfying $p(\Z)\subset \Z$. For example, we show that for every mildly mixing minimal system $(X,T)$ and all non-empty open subsets $U, V_1, \dots, V_d$ of $X,$
\begin{equation}\label{a}
\{n\in \Z: U\cap T^{-\lel\pi n\lel \sqrt{5}n^2+2n \rr \rr}V_1 \cap\dots \cap T^{-d\lel\pi n\lel \sqrt{5}n^2+2n \rr \rr}V_d\neq \emptyset\}
\end{equation} 
is an IP$^*$-set, where $\lel a\rr$ ($a\in \R$) means  the minimum integer which is the nearest integer to $a$. More generally, the role of the sequence $\lel\pi n\lel \sqrt{5}n^2+2n \rr \rr$ in (\ref{a}) can be played by any integer-valued generalized polynomial under suitable hypotheses. The set of \emph{integer-valued generalized polynomials} $\Z\to \Z$ is the smallest set $\mathcal{G}$ that is a function algebra (i.e., is closed under sum and products) containing $\Z[x]$ and having the additional property that for all $k\in \N$, $a_1,\dots, a_k\in\R$ and $p_1, \dots, p_k\in \mathcal{G},$ the mapping $n\to \lel \sum_{i=1}^{k}a_ip_i(n)\rr$ is in $\mathcal{G}.$ Generalized polynomials have been studied extensively, see for example the remarkable paper by Bergelson and Leibman \cite{BL2007} and references therein. In this paper we note that our definition of the generalized polynomials is slightly different from the usual one.
%sequence $g\in \mathcal{G}$, where  $\mathcal{G}$ is the set of generalized polynomials.  
%called integer-valued generalized polynomial. 

In the other hand, measurable dynamics and topological dynamics are two sister branches of the theory of dynamical systems that use similar words to describe different but parallel notions in their respective theories. For a measurable preserving system $(X,\mathcal{B},\mu,T)$, where $T$ is an invertible measurable preserving transformation acting on a probability space $(X,\mathcal{B},\mu),$ which is called \emph{mildly mixing} if for any $A, B\in \mathcal{B}$ and any $\ep>0$ the set of $n$ for which $|\mu(A\cap T^{-n}B)-\mu(A)\mu(B)|<\ep$ is an IP$^*$-set. In \cite[Section 9.5]{F81}, Furstenberg showed that mild mixing implies mild mixing of all orders. Later on, Bergelson stated in \cite[Theorem 4.8]{Bergelson87}, establishing a polynomical extension of Furstenberg's, as an unproved corollary to (the proof of) his main result. In 2009, McCutcheon formulated and proved the following result.

\medskip

\noindent {\bf Theorem}\cite[Theorem C]{MQ2009}
{\em If $p_i, 1 \le i \le d$ are generalized polynomials,  such that no $p_i$ and no $p_i-p_j,$ $1\le i < j \le d,$ is constant on an IP-set, (A subset $E$ of $\Z$ is IP if there is a sequence $\{n_i\}\subset \Z$ such that for any non -empty $\a\subset \N$, $\sum_{i\in \a}n_i\in E$.) then for any mildly mixing system $(X,\mathcal{B},\mu, T)$ and any $A_0, A_1, \dots , A_d \in \mathcal{B},$ the set $\{n\in \Z : |\mu(A_0\cap T^{-p_1(n)}A_1\cap \dots \cap T^{-p_d(n)}A_d)-\mu(A_0)\mu(A_1)\dots \mu(A_d)| < \ep \}$ is an IP$^*$-set.}

\medskip

In this paper we are interested in obtaining the topological version of the theorem above.  The next theorem is the main result of this paper.

\begin{thm} \label{thm general}
	Let $(X, T)$ be a topologically mildly mixing minimal system, $p_1, \dots, p_d$ be integer-valued generalized polynomials with the property that  $(p_1, \dots,p_d)$ is non-degenerate (see Definition \ref{def-nondegenerate}). Then for all non-empty open subsets $U , V_1, \dots, V_d $ of $X$,
	$$\{n\in \Z: U\cap T^{-p_1(n) }V_1\cap \dots \cap T^{-p_d(n) }V_d \neq \emptyset \}$$
	is an IP$^*$-set.

\end{thm}

As a corollary of Theorem \ref{thm general} (see Lemma \ref{equivalent:dense-delta transitive} for the reason), we have:
	
\begin{cor}
	Let $(X,T)$ be a topologically mildly mixing minimal system and $d\in \N$. Let $p_1,\dots, p_d$ be integer-valued generalized polynomials with the property that  $(p_1, \dots,p_d)$ is non-degenerate. Then for every IP-set $A$, there is a dense $G_\delta$ subset $X_0$ of $X$ such that for any $x\in X_0$,
	$$\{(T^{p_1(n)}x, \dots, T^{p_d(n)}x): n \in A\}$$
	is dense in $X^d$.
\end{cor}

The paper is organized as follows. In Section 2, we provide some basic notations, definitions and results  which will be used later.
In Section 3, we prove Theorem \ref{thm general} for integer-valued generalized polynomials of degree $1.$
In Section 4, we show the generalized polynomial extension of van der Waerden's theorem which is an important ingredient of our method.
In Section 5, we give  the  complete proof of the main result Theorem \ref{thm general}.
 
%\bigskip

\section{Preliminaries}

%In this section we will provide some basic notations, definitions and results  which will be used later.

\subsection{Subsets of integers}\
\medskip

In this paper, the set of integers and the set of positive integers are denoted by $\mathbb{Z}$
and $\mathbb{N}$ respectively.
Let $\F$ denote the family of all non-empty finite subsets of $\N,$
 i.e., $\a\in \F$ iff $\a=\{i_1,i_2,\dots, i_k\}\subset \N$, $i_1<i_2<\dots
 <i_k$ for some $k\in \mathbb{N}$.  
For $\a, \b\in \F$, we denote $\a<\b$ (or $\b>\a$) if $\max
\a<\min \b$.

Let $\{n_i\}_{i=1}^\infty$ be a sequence of $\Z$, define
$$FS(\{n_i\}_{i=1}^{\infty}) = \{n_{i_1}+n_{i_2}+\dots +n_{i_k}:
i_1<i_2<\dots <i_k, \  k\in \N \}.$$ 
Note that we do not require the elements of $\{n_i\}_{i=1}^{\infty}$ to be distinct.
We denote $n_\a=\sum_{i\in \a}n_i$, then
$$FS(\{n_i\}_{i=1}^{\infty})=\{n_\a: \a \in \F\}.$$
An {\em IP-set} is a set containing $\{n_{\alpha}: \alpha\in \mathcal{F} \}$ for some infinite sequence $\{n_i\}_{i=1}^\infty \subseteq \Z.$
Note that for any $\alpha_0\in \mathcal{F}$, the set $\{n_{\alpha}: \alpha> \alpha_0, \alpha \in \mathcal{F}\}$ is still an IP-set.
A set is called an {\em IP$^*$-set} if it intersects any IP-set. 
Let $\mathcal{F}_{IP}$ and $\F_{IP}^*$ denote the family of IP-sets and IP$^*$-sets respectively.
It is well known that the following lemmas hold (see, e.g., \cite{F81,F84}).

\begin{lem}\label{IP+IPstar}
	$\F_{IP}^*$ is a filter, i.e., $A_1, A_2\in \F_{IP}^*$ implies that $A_1\cap A_2 \in \F_{IP}^*;$ for all $A\in \mathcal{F}_{IP}$ and $B\in \F_{IP}^*,$ we have $A\cap B\in \mathcal{F}_{IP}.$
\end{lem}

\begin{lem}\label{divisible}
	Let $m\neq 0$ be an integer. Any IP-set contains an IP-set consisting of integers divisible by $m.$
\end{lem}

%$\mathcal{F}_{IP*}$ is a filter, i.e., $A_1, A_2\in \mathcal{F}_{IP*}$ implies that $A_1\cap A_2 \in \mathcal{F}_{IP*};$ and the intersection of an IP-set and an IP$^*$-set is an IP-set.

\subsection{Partition regular sets} \
\medskip

Bergelson, Hindman and Kra \cite{BHK96} introduced the partition regular sets in $\mathbb{N}$.
We now extend it to $\mathbb{Z}$, and generalize some results to more general cases.

Let $\mathcal{B}$ be a set of subsets of $\Z.$ 
We say that $\mathcal{B}$ is \emph{partition regular} if whenever $A_1 \cup\dots \cup A_n\in \mathcal{B}$ one has some $A_i\in \mathcal{B}$ with $i\in\{1,\dots,n\}.$  

\medskip

The following well known theorem shows that $\mathcal{F}_{IP}$ is partition regular.
\begin{thm}\cite[Theorem 1.4]{F84}\label{Theorem 1.4}
	Let $A$ be an IP-set. If $A$ is partitioned into finitely many sets, $A=C_1 \cup \dots\cup C_r,$ then some $C_j$ is an IP-set with $j \in \{1, \dots, r\}$.
\end{thm}
 
%The following property of IP$^*$-sets is readily deduced from the definition and from Theorem \ref{Theorem 1.4}.
%
%\begin{lem}\label{IP+IPstar}
%	The intersection of an IP$^*$-set and an IP-set is an IP-set. The intersection of finitely many IP$^*$-sets is an IP$^*$-set.
%\end{lem}

Our one aim is to generalize \cite[Lemma 2.2]{BHK96} (see Lemma \ref{lemma:partition regular}), and for that we need to recall some notations.

\medskip 

For a real number $a\in \mathbb{R}$, let $||a||=\inf\{|a-n|: n\in \mathbb{Z}\}$ and 
$$\lel a \rr=\inf\{m\in \mathbb{Z}: |a-m|=||a||\}.$$
We put $\{a\}=a-\lel a\rr$, then $\{a\}\in (-\frac{1}{2}, \frac{1}{2}]$.

We denote $[a]$ the greatest integer not exceeding $a$ and put $\{\{ a\}\}=a-[a]$, then $\{\{a\}\}\in [0, 1)$.

We have the following observations: 
\begin{lem} \label{finite sum}
	Let $k\in \mathbb{N}$ and $r_1, \dots, r_k\in \mathbb{R}$. If 
	\begin{enumerate}
		\item If $-\frac{1}{2}< \{r_1\}+\dots + \{r_k\}\leq \frac{1}{2},$ 
	         then $\lel r_1 +\dots +r_k \rr =\lel r_1\rr+ \dots + \lel r_k \rr.$ 
		\item If $0\le \{\{r_1\}\}+\dots + \{ \{r_k \}\}< 1,$ 
		      then $[r_1+\dots + r_k] =[r_1] + \dots + [r_k].$ 
	\end{enumerate}
	
\end{lem}

%The following lemma is a generalization of \cite[Lemma 2.2]{BHK96}.

\begin{lem} \label{lemma:partition regular}
	Let $\mathcal{B}$ be a set of subsets of $\Z $ such that
	\begin{enumerate}
		\item  \label{partition regular}$\mathcal{B}$ is partition regular.
		\item  \label{sum} For each $A\in \mathcal{B}$, there exists $x, y\in A$ with $x+y\in A$.
	\end{enumerate}
	Then for each $b_i, \alpha_i, c_j, \beta_j \in \mathbb{R}$, $ i=1, \dots, t_1, j=1, \dots, t_2$, $t_1, t_2\in \mathbb{N}$, any $\epsilon>0$ and every $A\in \mathcal{B}$,
	there exists $B\in \mathcal{B}$ such that $B\subset A$ and for all $n\in B$, 
	$$\{ b_i\lel \alpha_i n\rr\}, \{\alpha_in\}, \{ c_j\lel \beta_j n\rr\} \ \text{and} \ \{\beta_jn\} \in (-\epsilon, \epsilon) $$
	for all $i=1, \dots, t_1$ and $j=1, \dots, t_2$.
\end{lem}
\begin{proof}
	Let $b_i, \alpha_i, c_j, \beta_j \in \mathbb{R}$, $ i=1, \dots, t_1, j=1, \dots, t_2$, $t_1, t_2\in \mathbb{N}$, $\epsilon>0$ and $A\in \mathcal{B}$ be given.
	Pick $m\in \mathbb{N}$ such that $\frac{1}{2m}<\epsilon$ and $m> 2$.
	
	For each $i=1, \dots, t_1$, put
	$$C_1^i=\{n\in A: -\frac{1}{2m} <  \{\alpha_i n\} \le \frac{1}{2m}  \},$$
	$$D^i_1=\{n\in A: -\frac{1}{2m} < \{b_i \lel \alpha n_i \rr\} \le \frac{1}{2m}\},$$
	$$C^i_k=\{n\in A: \frac{k-1}{2m} <  \alpha_i n  -  [ \alpha_i n ]  \le \frac{k}{2m}\}, $$
	$$D^i_k=\{n\in A: \frac{k-1}{2m} < b_i \lel \alpha_i n \rr - [ b_i \lel \alpha_i n \rr ] \le  \frac{k}{2m}\},$$
	where $k\in \{2, 3, \dots, 2m-1\}$.
	And for each $j=1, \dots, t_2$, put
	
	$$E_1^j=\{n\in A: -\frac{1}{2m} < \{\beta_j n\} \le \frac{1}{2m}  \},$$
	$$F^j_1=\{n\in A: -\frac{1}{2m} < \{c_j \lel \beta_j c_j \rr\} \le \frac{1}{2m}\},$$
	$$E^j_k=\{n\in A: \frac{k-1}{2m} <  \beta_j n  -  [\beta_j n ]  \le \frac{k}{2m}\},   $$
	$$F^j_k=\{n\in A: \frac{k-1}{2m} < c_j \lel \beta_j n \rr - [ c_i \lel \beta_i n \rr ] \le \frac{k}{2m}\},$$
	where $k\in \{2, 3, \dots, 2m-1\}$.
	
	Put $$A_1=\bigcap_{i=1}^{t_1} C_1^i \cap \bigcap_{i=1}^{t_1} D_1^i \cap \bigcap_{j=1}^{t_2}E_1^j \cap \bigcap_{j=1}^{t_2}F_1^j $$
	and 
	$$A_{(k^C, k^D, k^E, k^F)}=\bigcap_{i=1}^{t_1} C^i_{k^C} \cap \bigcap_{i=1}^{t_1} D^i_{k^D} \cap \bigcap_{j=1}^{t_2} E^j_{k^E} \cap \bigcap_{j=1}^{t_2} F^j_{k^F}$$
	where $k^C, k^D, k^E, k^F\in \{1, \dots, 2m-1\}$, and $A_{(1, 1, 1, 1)}=A_1$.
	
	Clearly,
	$$A=\bigcup_{k^C, k^D, k^E, k^F=1, \dots, 2m-1 }A_{(k^C, k^D, k^E, k^F)}.$$
	
	By condition (\ref{partition regular})  we have $A_{(k^C, k^D, k^E, k^F)}\in \mathcal{B}$ for some $k^C, k^D, k^E, k^F \in \{1, \dots, 2m-1\}$.
	We will show that such $k^C, k^D, k^E, k^F $ satisfy that $k^C=k^D=k^E= k^F=1$.
	
	\begin{claim} \label{CE=1}
		$k^C=1$ and $k^E=1$.
	\end{claim}
	\begin{proof}
		Firstly, we show that $ k^C=1.$ Let $x$ and $y$ be two elements of $A_{(k^C, k^D, k^E, k^F)}$, then for all $i=1, \dots, t_1$, one has $x, y\in  C^i_{k^C}$.
		
		Fix $i=1, \dots, t_1$.
		If $2\le k^C <  m-1 $, then $x+y\in C^i_{2k^C-1} \cup C^i_{2k^C}$.
		If $k^C=m$, then $x+y\in C^i_{2m-1} \cup C^i_{1}$.
		If $m+1 \le  k^C \leq 2m-1$, then $x+y\in C^i_{2(k^C-m)-1} \cup C^i_{2(k^C-m)} $.
		In any of these cases $x+y\notin C^i_{ k^C },$ so $x+y\notin A_{(k^C, k^D, k^E, k^F)}$and hence by condition (\ref{sum}) we must have $ k^C=1$.
		
		By the same method, we have $k^E=1$.
		
	\end{proof}
	
	\begin{claim}
		$k^D=1$ and $k^F=1$.
	\end{claim}
	\begin{proof}
		Firstly , we show that $k^D=1$.
		Since $\frac{1}{2m}<\frac{1}{2}$ and by Claim \ref{CE=1} and Lemma \ref{finite sum}, then for all $x,y\in A_{(1, k^D, 1, k^F)}$ we have
		\begin{equation} \label{2-sum}
		\lel \alpha_i x \rr +\lel \alpha_i y \rr=\lel \alpha_i (x+y) \rr
		\end{equation}
		for any  $i=1, \dots, t_1$.
		
		Let $x, y \in A_{(1, k^D, 1, k^F)} $, then for all $i=1, \dots, t_1$, one has $x, y\in  D^i_{k^D}$.
		
		Now fix $i=1, \dots, t_1$. If $2\le k^D \le m-1$, that is $$\frac{ k^D-1}{2m} <b_i \lel \alpha_i x \rr - [ b_i \lel \alpha_i x \rr ] \le \frac{ k^D}{2m}$$
		and 
		$$\frac{ k^D-1}{2m} < b_i \lel \alpha_i y \rr - [ b_i \lel \alpha_i y \rr ] \le  \frac{ k^D}{2m}.$$
		Then $0<\{\{  b_i \lel \alpha_i x \rr    \}\}+ \{\{  b_i \lel \alpha_i y \rr    \}\} <1 $, thus by Lemma \ref{finite sum} and $(\ref{2-sum})$, we have
		$$[ b_i \lel \alpha_i x \rr ] +[ b_i \lel \alpha_i y \rr ] = [b_i \lel \alpha_i x \rr+b_i \lel \alpha_i y \rr  ]= [ b_i \lel \alpha_i (x+y) \rr ],$$ 
		which shows that $x+y \in D^i_{2k^D-1}\cup D^i_{2k^D}$.
		By the same discussion, we have
		if $k^D=m$, then $x+y \in D^i_{2k^D-1 }\cup D_1^i$;
		if $m+1 \le k^D \le 2m-1$, then $x+y \in D^i_{2(k^D-m)-1} \cup D^i_{ 2(k^D-m)}$.
		In any of these cases $x+y\notin  A_{(1, k^D, 1, k^F)}$, and hence by condition (\ref{sum}) we must have $k^D=1$.
		
		By the same method, we have $k^F=1$.
		
	\end{proof}
	
	So far, for all $n\in A_{(1,1,1,1)}\in \mathcal{B}$ we have  
	$$\{ b_i\lel \alpha_i n\rr\}, \{\alpha_in\}, \{ c_j\lel \beta_j n\rr\} \ \text{and} \ \{\beta_jn\} \in (-\epsilon, \epsilon) $$
	for all $i=1, \dots, t_1$ and $j=1, \dots, t_2$. 
	
\end{proof}

\begin{rem}\label{the hypotheses}
	$\mathcal{F}_{IP}$ satisfies the hypotheses of Lemma \ref{lemma:partition regular}.
\end{rem}

\begin{lem}\label{lemma:spectra of IP-set is IP 1}
	Let $p(n)=\sum\limits_{i=1}^{\rm{t_1}} \left\lceil b_i\left\lceil \alpha_i n\right\rceil \right\rceil-\sum\limits_{j=1}^{\rm{t_2}} \left\lceil c_j\left\lceil \beta_j n\right\rceil \right\rceil$ with $t_1,t_2\in \N,\ b_i, \alpha_i, c_j, \beta_j \in \mathbb{R},\ i=1, \dots, t_1$, $j=1, \dots, t_2 $ and $$\sum_{i=1}^{t_1}b_i\alpha_i-\sum_{j=1}^{t_2}c_j\beta_j \neq 0.$$
	If $A$ is an IP-set, then
	 $\{p(n): n\in A \}$ is also an IP-set. 
\end{lem}

\begin{proof}
 
	Let $\epsilon=\frac{1}{4}$. By Lemma \ref{lemma:partition regular} and Remark \ref{the hypotheses},  there exists an IP-set $B\subset A$ such that for all $n\in B$, we have
	$$\{ b_i\lel \alpha_i n\rr\}, \{\alpha_in\}, \{ c_j\lel \beta_j n\rr\} \ \text{and} \  \{\beta_jn\} \in (-\epsilon, \epsilon) $$
	for all $i=1, \dots, t_1$ and $j=1, \dots, t_2$.

	Pick the sequence $\{a_n\}_{n=1}^{\infty}$ of $\Z$ with $FS(\{a_n\}_{n=1}^{\infty})\subseteq B$.
	We show by induction on $|F|$ that for a finite subset $F$ of $\mathbb{N},$
	\begin{align*}
		&\sum_{n\in F}\biggl(\sum\limits_{i=1}^{\rm{t_1}} \lel b_i\lel \alpha_i a_n\rr \rr-\sum\limits_{j=1}^{\rm{t_2}} \lel c_j\lel \beta_j a_n\rr \rr\biggr) \\
		=&\sum\limits_{i=1}^{\rm{t_1}} \lel b_i\lel \alpha_i (\sum_{n\in F} a_n)\rr \rr-\sum\limits_{j=1}^{\rm{t_2}} \lel c_j\lel \beta_j (\sum_{n\in F} a_n)\rr \rr.
	\end{align*}
	This shows that $FS(\{ \sum\limits_{i=1}^{\rm{t_1}} \lel b_i\lel \alpha_i a_n\rr \rr-\sum\limits_{j=1}^{\rm{t_2}} \lel c_j\lel \beta_j a_n\rr \rr  \}_{n=1}^{\infty}) \subset \{ p(n): n\in A \}$.
	
	If $|F|=1$, the conclusion is trivial.	We now assume that $|F|>1$. Pick $m\in F$ and let $F'=F\backslash \{m\}$. By induction, we have
	\begin{align*}
		&\sum_{n\in F'}\biggl(\sum\limits_{i=1}^{\rm{t_1}} \lel b_i\lel \alpha_i a_n\rr \rr-\sum\limits_{j=1}^{\rm{t_2}} \lel c_j\lel \beta_j a_n\rr \rr\biggr) \\
		=&\sum\limits_{i=1}^{\rm{t_1}} \lel b_i\lel \alpha_i (\sum_{n\in F'} a_n)\rr \rr-\sum\limits_{j=1}^{\rm{t_2}} \lel c_j\lel \beta_j (\sum_{n\in F'} a_n)\rr \rr.
	\end{align*}
	Then 
	\begin{align*}
		&\sum_{n\in F} \biggl(\sum\limits_{i=1}^{\rm{t_1}} \lel b_i\lel \alpha_i a_n\rr \rr-\sum\limits_{j=1}^{\rm{t_2}} \lel c_j\lel \beta_j a_n\rr \rr\biggr) \\
		=& \biggl(\sum\limits_{i=1}^{\rm{t_1}} \lel b_i\lel \alpha_i a_m\rr \rr-\sum\limits_{j=1}^{\rm{t_2}} \lel c_j\lel \beta_j a_m\rr \rr\biggr) +
		\sum_{n\in F'}  \biggl(\sum\limits_{i=1}^{\rm{t_1}} \lel b_i\lel \alpha_i a_n\rr \rr-\sum\limits_{j=1}^{\rm{t_2}} \lel c_j\lel \beta_j a_n\rr \rr\biggr)\\
		=&\biggl(\sum\limits_{i=1}^{\rm{t_1}} \lel b_i\lel \alpha_i a_m\rr \rr-\sum\limits_{j=1}^{\rm{t_2}} \lel c_j\lel \beta_j a_m\rr \rr\biggr) \\
		& + \biggl(\sum\limits_{i=1}^{\rm{t_1}} \lel b_i\lel \alpha_i (\sum_{n\in F'} a_n)\rr \rr-\sum\limits_{j=1}^{\rm{t_2}} \lel c_j\lel \beta_j (\sum_{n\in F'} a_n)\rr \rr \biggr)\\
		=&\sum\limits_{i=1}^{\rm{t_1}} \lel b_i\lel \alpha_i (\sum_{n\in F} a_n)\rr \rr-\sum\limits_{j=1}^{\rm{t_2}} \lel c_j\lel \beta_j (\sum_{n\in F} a_n)\rr \rr .
	\end{align*}
	The last equation follows from that $a_m$, $\sum_{n\in F'}a_n$ are in $B$ and Lemma \ref{finite sum}.
	
\end{proof}

As a consequence of the above lemma, we have

\begin{cor}\label{frac{A}{q}}
	Let $q\in \Z. $ If $A$ is an IP-set, then $\{\lel\frac{n}{q}\rr:n\in A \}$ is an IP-set.
\end{cor}

% \begin{proof}
% 	Since $A$ is an IP-set, we pick $\{n_i\}_{i=1}^{\infty}$ with $FS(\{n_i\}_{i=1}^{\infty})\subseteq A.$ 
% 	By the pigeonhole principle, one may assume that $n_i=n_j$ (mod $q$) for all $i,j\in \N.$ Let $m_i=\sum_{t=(i-1)q+1}^{iq}n_t$ one has $q$ divides each $m_i,$ so each $\frac{m_i}{q}\in \Z.$ 
% 	Let $F$ be any finite nonempty subset of $\N$. 
% 	Then $\sum_{i\in F}\frac{m_i}{q}=\lel (\sum_{i\in F}m_i)\cdot\frac{1}{q}\rr$ and $\sum_{i\in F}m_i \in A.$ 
% 	Hence we have $FS(\{\frac{m_i}{q}\}_{i=1}^{\infty})\subseteq \{\lel \frac{n}{q}\rr:n\in A \},$ 
% 	so $\{\lel \frac{n}{q}\rr:n\in A \}$ is an IP-set.
% 	
% \end{proof}
 
 And we have the following useful result.
 
 \begin{lem}\label{qA}
 	Let $q\in \Z. $ If $A$ is an IP$^*$-set, then $qA=\{qn:n\in A\}$ is an IP$^*$-set.
 \end{lem}
 
 \begin{proof}
 	For any given IP-set $B,$ we will show that $qA\cap B\neq \emptyset.$ By Lemma \ref{divisible} one sees that there exists $B'\subseteq B$ with $B'$ is an IP-set and $B'\subseteq q\Z.$ Let $B''=\{\lel \frac{n}{q}\rr:n\in B'\}.$ By Corollary \ref{frac{A}{q}}, one has $B''$ is an IP-set. Pick $m\in B''\cap A$ and pick $n\in B'\subseteq B$ such that $m=\lel \frac{n}{q}\rr.$ Since $q$ divides $n$ we have $m=\frac{n}{q}$ so $n=qm.$ Thus $n\in qA\cap B.$ 
 	The proof is complete.
 \end{proof}

\subsection{Topological dynamics}\
\medskip

Let $(X, T)$ be a dynamical system.
For $x\in X$, we denote the {\it orbit} of $x$ by $\text{orb}(x, T)=\{T^{n}x \in X: n\in \mathbb{Z}\}$.	
A point $x\in X$ is called a \emph{transitive point} if the orbit of $x$ is dense in $X$,
i.e., $ \overline{\text{orb}(x, T)}=X$.
A dynamical system $(X, T)$ is called \emph{minimal} if every point $x\in X$ is a transitive point.
%Also,
%$x\in X$ is minimal piont iff $N(x, U)$ is a syndetic set for every neighborhood $U$ of $x$.

Let $U, V \subset X$ be two non-empty open sets,
the \emph{hitting time set} of $U$ and $V$ is denoted by
$$N(U, V)=\{n\in \mathbb{Z}: U \cap T^{-n}V \neq \emptyset\}.$$

We say that $(X, T)$ is \emph{(topologically) transitive} if for all non-empty open sets $U, V \subset X$, 
$N(U, V)$ is non-empty. Two dynamical systems are \emph{weakly disjoint} if their product is transitive. 
A dynamical system is called \emph{mildly mixing} if it is weakly disjoint from any transitive system. 

We say that $(X, T)$ is \emph{IP$^*$-transitive} if for all non-empty open sets $U, V \subset X$, $N(U, V)$ is an IP$^*$-set. When $(X, T)$ is a minimal system, it is mildly mixing if and only if it is IP$^*$-transitive (see \cite{HY2004} for more details).

Let $p_i: \mathbb{Z} \rightarrow \mathbb{Z},i=1, \dots,k$, we say that $(X,T)$ is \emph{$\{p_1, \dots,p_k\}$-IP$^*$-transitive} if
for any non-empty open sets $U_i, V_i \subset X,i=1, \dots, k$,
$$N(\{p_1, \dots, p_k\},U_1\times \dots \times U_k,V_1 \times \dots \times V_k ):=\bigcap_{i=1}^{k} N(p_i, U_i,V_i)$$		
is an IP$^*$-set, where	$N(p_i, U_i,V_i):=\{n\in \mathbb{Z}: U_i \cap T^{-p_i(n)}V_i \neq \emptyset\}$, $i=1, \dots,k$. $(X, T)$ is \emph{$\{p_1, \dots,p_d\}_{\D}$-IP$^*$-transitive} if for all non-empty open sets $U , V_1, \dots, V_d \subset X$, $$\{n\in \Z: U\cap T^{-p_1(n) }V_1 \cap \dots \cap T^{-p_d(n) }V_d \neq \emptyset \}$$ is an IP$^*$-set.

\medskip

The following lemma is a generalization of \cite[Lemma 5]{KO2012}. For completeness, we include a proof.

\begin{lem}\label{Main lemma:descending sequence}
	Let $(X, T)$ be a dynamical system, $d\in \mathbb{N}$ and maps $p_1, \dots, p_d: \mathbb{Z} \to \mathbb{Z}$ such that $(X, T)$ is $\{p_1, \dots, p_d\}$-IP$^*$-transitive. Then for all non-empty open subsets $V_1, \dots, V_d$ of $X$, for any IP-set   $A=\{n_{\alpha}:\alpha\in \mathcal{F}\}$ and any subsequence $\{ r(n)\}_{n=0}^{\infty}$ of natural numbers, there is an increasing sequence $\{\alpha_j\}_{j=0}^{\infty}\subset \mathcal{F}$ such that $|n_{\alpha_0}|> r(0), |n_{\alpha_j}|>|n_{\alpha_{j-1}}|+r(|n_{\alpha_{j-1}}|)$ for all $j \ge 1$, and for each $i\in \{1, \dots, d\}$,
	there is a descending sequence $\{V_i^{(n)}\}_{n=0}^{\infty}$ of open subsets of $V_i$ such that for each $n\ge 0$, one has
	$$T^{p_i(n_{\alpha_j})}T^{-j}V_i^{(n)}\subset V_i, \ \  \text{for all} \ \ 0\le j \le n.$$
	
%	Let $(X, T)$ be a t.d.s., $d\in \mathbb{N}$ and $p_1, \dots, p_d: \mathbb{Z} \to \mathbb{Z}$ such that $(X, T)$ is $\{p_1, \dots, p_d\}$-IP$^*$-transitive. 
%	Then for all non-empty open sets $V_1, \dots, V_d$ of $X$, and any IP-set $C=\{ n_{\alpha}\}_{\alpha\in \mathcal{F}}$,  there is an increasing sequence $\{\alpha_n\}_{n=0}^{\infty}\subset \mathcal{F}$ such that $|n_{\alpha_n}|>n$ and $n_{\alpha_n}\in C$, for all $n\ge 1$, and for each $i\in \{1, \dots, d\}$,
%	there is a descending sequence $\{V_i^{(n)}\}_{n=0}^{\infty}$ of open subsets of $V_i$ such that for each $n\ge 0$, one has
%	$$T^{p_i(n_{\alpha_j})}T^{-j}V_i^{(n)}\subset V_i, \ \  \text{for all} \ \ 0\le j \le n.$$
\end{lem}
\begin{proof}
	Let $V_1, \dots, V_d$ be non-empty open subsets of $X$, $A=\{n_{\alpha}: \alpha\in \mathcal{F}\}$ be an IP-set 
	and $\{ r(n)\}_{n=0}^{\infty}$ be a subsequence of $\N$.
	Since $(X, T)$ is $\{p_1, \dots, p_d\}$-IP$^*$-transitive,
	 $\bigcap_{i=1}^dN(p_i, V_i, V_i)$ 
	is an IP$^*$-set.
    Thus there exists $\alpha_0\in \mathcal{F}$ such that $$n_{\alpha_0}\in \bigcap_{i=1}^dN(p_i, V_i, V_i)\cap A$$ and $|n_{\alpha_0}|> r(0).$ 
    %as $A$ is an IP-set.
	%Let $k_0\in \bigcap_{i=1}^dN(p_i, V_i, V_i)\cap C $ such that $|k_0|>r(0)$.
	That is $T^{-p_i(n_{\alpha_0})}V_i \cap V_i \neq \emptyset$ for all $i=1, \dots, d$. Put $V_i^{(0)}=T^{-p_i(n_{\alpha_0})}V_i \cap V_i $ for all $i=1, \dots, d$ to complete the base step.
	
	Now assume that for $n\ge 1$ we have found a sequence $\alpha_0 <\alpha_1< \dots< \alpha_{n-1}\in \F$ and for each $i=1, \dots, d$,
	we have non-empty open subsets $V_i \supset V_i^{(0)} \supset V_i^{(1)} \supset \dots \supset V_i^{(n-1)}$ such that
	for each $m=0, 1, \dots, n-1$ one has $|n_{\alpha_m}|>|n_{\alpha_{m-1}}|+r(|n_{\alpha_{m-1}}|)$  and 
	$$T^{p_i(n_{\alpha_j})} T^{-j}V_i^{(m)} \subset V_i, \ \ \text{for all} \ 0\le j \le m.$$
	
	For $i=1, \dots, d$, let $U_i=T^{-n}V_i^{(n-1)}$. Then $\bigcap_{i=1}^dN(p_i, U_i, V_i)$ is an IP$^*$-set 
	since $(X, T)$ is $\{p_1, \dots, p_d\}$-IP$^*$-transitive,  
	thus there exists $\alpha_n\in \mathcal{F}$ such that $\alpha_n>\alpha_{n-1}$, $ |n_{\alpha_n}|>|n_{\alpha_{n-1}}|+r(|n_{\alpha_{n-1}}|)$ , and
	$$(U_1\times \dots \times U_d ) \cap T^{-p_1(n_{\alpha_n})} \times \dots \times T^{-p_d(n_{\alpha_n})}(V_1\times \dots \times V_d)\neq \emptyset.$$
	%	Hence $C\cap \bigcap_{i=1}^dN(p_i, U_i, V_i) $ is an IP-set.
	%	Then we choose $k_n\in C\cap \bigcap_{i=1}^d N(p_i, U_i, V_i)$ such that 
	%	$|k_n|> |k_{n-1}|+r(|k_{n-1}|)$.
	That is
	$$T^{-p_i(\alpha_n)}V_i\cap U_i \neq \emptyset, \ \text{for all} \ i=1, \dots, d.$$	
	Then for $i=1, \dots, d$,
	$$ T^{p_i(n_{\alpha_n})}U_i \cap V_i= T^{p_i(n_{\alpha_n})} T^{-n}V_i^{(n-1)}\cap V_i \neq \emptyset. $$
	Let
	$$V_i^{(n)}=V_i^{(n-1)} \cap (T^{p_i(n_{\alpha_n})} T^{-n})^{-1}V_i.$$
	Then $V_i^{(n)} \subset V_i^{(n-1)}$ is a non-empty open set and
	$$ T^{p_i(n_{\alpha_n})} T^{-n} V_i^{(n)} \subset V_i.$$
	Since $V_i^{(n)} \subset V_i^{(n-1)} $, we have
	$$T^{p_i(n_{\alpha_j})}T^{-j}V_i^{(n)} \subset V_i,  \ \text{for all} \ 0\le j\le n.$$
	Hence we finish the induction. The proof is completed. 
	
\end{proof}

\begin{lem}\label{equivalent:dense-delta transitive}
	Let $(X, T)$ be a dynamical system, $d\in \mathbb{N}$, $p_1, \dots, p_d: \mathbb{Z} \to \mathbb{Z}$ and let $A$ be a subset of $\Z$.
	Then the following assertions are equivalent.
	\begin{enumerate}
		\item \label{Lemma:1open sets}
		   If $U, V_1, \dots, V_d \subset X$ are non-empty open sets, then there exists $n\in A$ such that
		   $$U \cap T^{-p_1(n)}V_1 \cap \dots \cap T^{-p_d(n)}V_d \neq \emptyset.$$
		\item \label{Lemma:2dense}
		   There exists a dense $G_{\delta}$ subset $X_0\subset X$ such that for every $x\in X_0$,
		   $$\{(T^{p_1(n)}x, \dots, T^{p_d(n)}x): n\in A\}$$
		   is dense in $X^d$.
	\end{enumerate}	
\end{lem}
\begin{proof}
		The proof is similar to the proof in  \cite{Moothathu2010}. For completeness, we include a proof.
	
	$(\ref{Lemma:1open sets}) \Rightarrow (\ref{Lemma:2dense})$:
	Consider a countable base of open balls
	$\{B_k: k\in \mathbb{N}\}$ of $X$.
	Put
	$$X_0=\bigcap_{(k_1, \dots, k_d)\in \mathbb{N}^d} \bigcup_{n\in A} \bigcap_{i=1}^{d}T^{-p_i(n)}B_{k_i}.$$
	For any $(k_1, \dots, k_d)\in \mathbb{N}^d$,
	the set $\cup_{n\in A} \cap_{i=1 }^{d}T^{-p_i(n)}B_{k_i}$ is open, and is dense by $(\ref{Lemma:1open sets})$.
	Thus by the Baire category theorem, $X_0$ is a dense $G_{\delta}$ subset of $X$.
	By construction, for every $x\in X_0$, $$\{ (T^{p_1(n)}x, \dots, T^{p_d(n)}x): n\in A\}$$ is dense in $X^d$.
	
	$(\ref{Lemma:2dense}) \Rightarrow (\ref{Lemma:1open sets})$: Choose $x\in X_0\cap U$ and $n\in A$ such that
	$$(T^{p_1(n)}x, \dots, T^{p_d(n)}x) \in V_1\times \dots \times V _d, $$
	then $x\in U\cap T^{-p_1(n)}V_1\cap \dots \cap T^{-p_d(n)}V_d $.
\end{proof}

\subsection{Generalized polynomials}\
\medskip

In \cite{HSY2016}, Huang, Shao and Ye introduced the notions of $GP_d$ and $\mathcal{F}_{GP_d}$, 
we now recall their definitions.

\begin{de}
	Let $d\in \mathbb{N},$  the \emph{generalized polynomials} of degree $\leq d$ (denoted by $GP_d$) is defined as follows.   	For $d=1$, $GP_1$ is the smallest collection of functions from $\mathbb{Z}$ to $\mathbb{R}$
	containing $\{h_a: a\in \mathbb{R} \} $ with $h_a(n)=an$ for each $n\in \mathbb{Z}$, which is closed under taking $\left\lceil {\cdot } \right\rceil$, multiplying by a constant and finite sums.
	
	Assume that $GP_i$ is defined for $i<d$. Then $GP_d$ is the smallest collection of functions from
	$\mathbb{Z}$ to $\mathbb{R}$ containing $GP_i$ with $i<d$, functions of the forms
	$$a_0n^{p_0}\left\lceil {f_1(n) } \right\rceil \dots \left\lceil {f_k(n)} \right\rceil $$
	(with $a_0\in \mathbb{R}, p_0 \ge 0, k\ge 0, f_l\in GP_{p_l}, 1\le l \le k$ and $\sum_{l=0}^{k}p_l=d$),
	which is closed under taking $\left\lceil {\cdot } \right\rceil$, multiplying by a constant and finite sums.
	Let $GP=\bigcup_{i=1}^{\infty}GP_i$. Note that if $p\in GP$, then $p(0)=0$.
\end{de}
\begin{de}
	Let $\mathcal{F}_{GP_d}$ be the family generated by the sets of the form
	$$\bigcap_{i=1}^{k}\{n\in \mathbb{Z}: p_i(n) \ ( \text{mod}\ \mathbb{Z})\ \in (-\varepsilon_i, \varepsilon_i)\},$$
	where $k\in \mathbb{N}$, $p_i\in GP_d$, and $\varepsilon_i>0, 1\le i\le k$.
	Note that $ p_i(n) \ ( \text{mod} \ \mathbb{Z}) \in (-\varepsilon_i, \varepsilon_i)$
	if and only if $\{p_i(n)\} \in (-\varepsilon_i, \varepsilon_i)$.

\end{de}
\begin{rem}\label{rem-filer}
	$\mathcal{F}_{GP_d}$ is a filter, i.e., $A_1, A_2\in \mathcal{F}_{GP_d}$ implies that $A_1\cap A_2 \in \mathcal{F}_{GP_d}$.
\end{rem}

A subset $A\subset \mathbb{Z}$ is a \emph{Nil$_d$ Bohr$_0$-set} if there exists a $d$-step nilsystem $(X, T)$, $x_0\in X$ and an open set $U\subset X$ containing $x_0$ such that
$N(x_0, U):= \{n\in \mathbb{Z}: T^nx_0 \in U\}$ is contained in $A$.
Denote by $\mathcal{F}_{d, 0}$ the family consisting of all Nil$_d$ Bohr$_0$-sets.  
A subset $A\subset \mathbb{Z}$ is called \emph{Nil Bohr$_0$-set} if $A\in \mathcal{F}_{d, 0}$ for some $d\in \mathbb{N}$.    	
In \cite{HSY2016}, the authors proved the following theorem.
\begin{thm}\cite[Theorem B]{HSY2016}\label{Bohr}
	Let $d\in \mathbb{N}$. Then $ \mathcal{F}_{d, 0}=\mathcal{F}_{GP_d}$.
\end{thm}
\begin{rem}
	Since a nilsystem is distal, it is shown in \cite[Theorem 9.11]{F81} that every Nil$_d$ Bohr$_0$-set is an IP$^*$-set. Together with  Remark \ref{rem-filer} we know $\mathcal{F}_{GP_d}$ is a filter and any $A \in \mathcal{F}_{GP_d}$ is an IP$^*$-set.
\end{rem}

 Zhang and Zhao \cite{ZZ2021} introduced the notions of integer-valued generalized polynomials and the special integer-valued generalized polynomials.
In this paper, we adopt their relevant definitions and list some results given by Zhang and Zhao 
(see \cite{ZZ2021} for more details).

\begin{de}
	For $d\in \mathbb{N}$, {\it the integer-valued generalized polynomials} of degree $\le d$ is defined by
	$$\widetilde{GP_d}=\{\left\lceil p(n) \right\rceil: p(n)\in GP_d\},$$
	and {\it the integer-valued generalized polynomials}  is then defined by
	$$\mathcal{G}=\bigcup_{i=1}^{\infty}\widetilde{GP_i}.$$
\end{de}
For $p(n) \in \mathcal{G}$,
the least $d\in \mathbb{N}$ such that $p\in \widetilde{GP_d}$ is defined by the \emph{degree} of $p$, denoted by
$\deg(p)$. 
 
\medskip
 
%Since the integer-valued generalized polynomials are very complicated, we will also
%specify a subclass of the integer-valued generalized polynomials, i.e. {\it the special integer-valued generalized polynomials} (denoted by $\widetilde{SGP}$),
%which will be used in the proof of our main results. 

%We need to recall the defintion of $L(a_1,a_2,\dots,a_l)$ in Defintion 4.2 of  \cite{HSY2016}. 

Since the integer-valued generalized polynomials are very complicated, we will also
specify a subclass of the integer-valued generalized polynomials, i.e. {\it the special integer-valued generalized polynomials} (denoted by $\widetilde{SGP}$),
which will be used in the proof of our main results. 

We  need to recall the defintion of $L(a_1,a_2,\dots,a_l)$ in Defintion 4.2 of  \cite{HSY2016}. For $a \in \mathbb{R}$, we define $L(a)=a$. For $a_1,a_2 \in \mathbb{R}$, we define $L(a_1,a_2)=a_1\left\lceil L(a_2)\right\rceil$. 
 Inductively, for $a_1,a_2, \dots, a_l \in \mathbb{R}(l \ge 2)$,
$$L(a_1,a_2,\dots,a_l)=a_1\left\lceil L(a_2,\dots,a_l)\right\rceil.$$ 

Before introducing the definition of $\widetilde{SGP}$,
we recall the notion of the simple generalized polynomials.

\begin{de} 	
	For $d \in \mathbb{N}$, the {\it simple generalized polynomials} of degree $ \le d$ \  (denoted by $\widehat{SGP_d}$\ ) is defined as follows.
	$\widehat{SGP_d}$ is the smallest collection  of generalized polynomials of the forms
	$$\prod_{i=1}^{k}L(a_{i,1}n^{j_{i,1}},\dots, a_{i,l_{i}}n^{j_{i,l_i}}),$$
	where $k \ge 1$, $1 \le l_i \le d$,  $a_{i,1},a_{i,2},\dots,a_{i,l_{i}} \in \mathbb{R}$, $j_{i,1},j_{i,2},\dots, j_{i,l_i} \ge 0$ and $\sum_{i=1}^{k}\sum_{t=1}^{l_i}j_{i,t} \le d$.
\end{de}

%For $g_1, \dots, g_t\in \widehat{SGP}_d$ and $\delta >0$, we denote
%$C(\delta, g_1, \dots, g_t)$ by
%$$ \bigcap_{k=1}^{t} \{n\in  \mathbb{Z}: \{q_k(n)\}\in (-\delta,\delta)\}. $$

%With the help of the above definition, we can intoduce the notion of  special integer-valued generalized polynomials.
\begin{de} 	For $d \in \mathbb{N}$, the {\it special integer-valued generalized polynomials} of degree $ \le d$ (denoted by $\widetilde{SGP_d}$) is defined as follows.
	$$\widetilde{SGP_d}=\{\sum_{i=1}^k c_i\left\lceil p_i(n)\right\rceil: p_i(n) \in \widehat{SGP_d} \text{ and } c_i \in \mathbb{Z}\}.$$
	The  {\it special integer-valued generalized polynomials}  is then defined by
	$$\widetilde{SGP}=\bigcup_{d=1}^{\infty} \widetilde{SGP_d}.$$
	%Clearly $\widetilde{SGP}\subset \mathcal{G} \subset GP$.
\end{de}
Clearly $\widetilde{SGP} \subset \mathcal{G}$.

The following lemma shows the relationship 
between $\widetilde{GP_d}$ and $\widetilde{SGP_d}$.

\begin{lem}\label{lem sgp_gp}\cite[Lemma 2.12]{ZZ2021}
	Let $d \in \mathbb{N}$ and $p(n) \in \widetilde{GP_d}$. Then there exists $h(n)\in \widetilde{SGP_d}$ and a set
	$$C=C(\delta,q_1,\dots,q_t)=\bigcap_{k=1}^{t} \{n\in  \mathbb{Z}: \{q_k(n)\}\in (-\delta,\delta)\}$$
	such that $$p(n)=h(n), \forall n \in C,$$ where $\delta>0$ is small enough and
	$q_k \in \widehat{SGP_d}, k=1,\dots, t$ for some $t \in \mathbb{N}$.
	
\end{lem}

\medskip

By Lemma \ref{lem sgp_gp}, every $p(n) \in \widetilde{GP_d}$ corresponds to an $h(n) \in \widetilde{SGP_d},$ the maximal-degree components of $p(n)\in \widetilde{GP_d}$ is denoted by the maximal-degree components of the corresponding  $h(n)\in \widetilde{SGP_d}.$

\begin{de}\label{def-nondegenerate}	
	Let $p(n) \in \mathcal{G}$ and $A(p(n))$ be the sum of the coefficients of the maximal-degree components of $p(n)$.  
	Let $p_1, \dots, p_d \in  \mathcal{G}$, a tuple $(p_1, \dots,p_d)$ is called a
	{\it non-degenerate tuple} if $A(p_i) \neq 0$ and $A(p_i-p_j) \neq 0$, $1 \le i\neq j \le d$. 
\end{de}

For instance, $A(\left\lceil an^2 \left \lceil bn\right\rceil +\left\lceil cn^3\right\rceil \right\rceil +dn^3 +2n^2)=ab+c+d,$ $A(n+n\left \lceil 2\pi n -\left\lceil 2\pi n\right\rceil \right\rceil)=2\pi -2\pi=0,$  $(n^2+n,n^2+\left\lceil \sqrt{3}n\right\rceil)$ is non-degenerate, $(n\left\lceil 2\pi n \right\rceil+n, \left\lceil 2\pi n^2 \right\rceil+2n)$ is not non-degenerate.

\medskip

%The key ingredient in the proof of the main result is to view the integer-valued generalized polynomials, in some sense, as the ordinary polynomials. To do this, we need to introduce the following definition.

\begin{de}\label{def proper}
	Let $p(n)\in \widetilde{SGP}$, $m \in \mathbb{Z}$ and $C \subset \mathbb{Z}$. We say that $p$ is \emph{proper} with respect to (w.r.t. for short) $m$ and $C$ if
	for every $n\in C$,
	\begin{itemize}
		\item if $\deg(p)=1$, $p(n+m)=p(n)+p(m)$.
		\item if $\deg(p)>1$, 	$p(n+m)-p(n)-p(m)=q(n),$
		where $q(n)\in \widetilde{SGP}$ and $\deg(q) <\deg(p)$.
	\end{itemize}
\end{de}

%For example, let $p(n)=\left\lceil an^2\right\rceil$, if
%$$p(n+m)=\left\lceil a(n+m)^2\right\rceil= \left\lceil an^2\right\rceil+\left\lceil am^2\right\rceil+\left\lceil 2amn\right\rceil, \forall n \in C,$$ then we say $p(n)$ is proper w.r.t. $m$ and $C$.

%Let $p(n) \in \widetilde{SGP}$, $m\in \mathbb{Z}$. To study whether there exists $C$ such that $p(n)$ is proper w.r.t. $m$ and $C$,  we need to introduce the following notion.
\begin{de}\label{def-m-good}
	Let $p(n) \in \widetilde{SGP}$ and $m\in \mathbb{Z}$. 
	\begin{itemize}
		\item If $p(n)=\left\lceil L(a_1n^{j_1},\dots, a_ln^{j_l}) \right \rceil$, we say $m$ is {\it good} w.r.t. $p(n)$ if for any $1\le t \le l$,  $\{L(a_{t}m^{j_{t}},a_{t+1}m^{j_{t+1}},\dots,a_lm^{j_l})\} \neq \frac{1}{2} $.
		\item If $p(n)=\left\lceil \prod_{i=1}^{k}r_{i}(n) \right\rceil$
		with $r_i(n)=L(a_{i,1}n^{j_{i,1}},\dots, a_{i,l_i}n^{j_{i,l_i}})$, we say $m$ is {\it good} w.r.t. $p(n)$ if	
		$\{\prod_{i=1}^{k}r_i(m)\} \neq \frac{1}{2}$ 	
		and  $m$ is
		{\it good} w.r.t. $\left\lceil r_i(n) \right \rceil$ for each $1 \le i \le k$.
		\item If $p(n)=\sum_{t=1}^{k}c_t\left\lceil q_t(n) \right\rceil$ with $c_t \in\mathbb{Z}$ and each $q_t(n)$ is of the form $\prod_{i=1}^{k}r_{i}(n)$ with $r_i(n)=L(a_{i,1}n^{j_{i,1}},\dots, a_{i,l_i}n^{j_{i,l_i}})$, we say $m$ is good
		w.r.t. $p(n)$ if $m$ is {\it good} w.r.t. $\left\lceil q_t(n)\right\rceil$ for each $1 \le t \le k$. 
		
	\end{itemize}
\end{de}
%For example, if $\{ bm\left\lceil cm \right\rceil\} \neq \frac{1}{2}$ and $\{ cm \} \neq \frac{1}{2}$,   then $m$ is good w.r.t. $p(n)=\left\lceil bn\left\lceil cn \right\rceil\right\rceil$.

The following  lemmas  are very useful in our proof. 
\begin{lem}\cite[Lemma 2.16]{ZZ2021}\label{lem-goodm}
	Let $p(n) \in \widetilde{SGP}$. Then there exist $\delta>0$, $Q\subset \widehat{SGP_s}$ (for some $s\in \mathbb{N}$) and
	$$C(\delta,Q)=\bigcap_{q(n)\in Q}\{n\in \mathbb{Z}: \{q(n)\} \in (-\delta,\delta)\}$$
	such that for each $m\in C(\delta,Q)$, $m$ is good w.r.t. $p(n)$.
\end{lem}

\begin{lem}\cite[Lemma 2.19]{ZZ2021}\label{lem-proper}
	Let $p_1, \dots, p_d\in \widetilde{SGP}$. Let $l\in \mathbb{N}, m_j\in \mathbb{Z}$ and $m_j$ is good w.r.t. $p_i(n)$ for $1\le i \le d, 1 \le j \le l$.  
	Then there exists a Nil$_s$ Boh$_0$-set $C$
	with the form
	$$C=\bigcap_{k=1}^{t}\{n\in \mathbb{Z}: \{q_k(n)\}\in (-\delta,\delta)\}$$
	such that for all $(i, j)\in \{1, \dots, d\} \times \{1, \dots, l\}$, $p_i(n)$ is proper w.r.t. $m_j$ and $C$,	where  $\delta>0$ is a small enough number,
	$s = \mathop {\max }_{1\le i \le d }\deg (p_i)$ and	$q_k \in \widehat{SGP_s}, k=1,\dots,t$ for some $t\in \mathbb{N}$.
\end{lem}

The Nil$_s$ Bohr$_0$-set $C$ above is called \emph{associated} to $\{p_1, \dots, p_d\}$ and $\{m_1, \dots, m_l\}$.

%\begin{rem} \label{rem}
%	We call the Nil$_s$ Bohr$_0$-set $C$ above is associated to $\{p_1, \dots, p_d\}$ and $\{m_1, \dots, m_l\}$.
%	%	We call $\delta$ above is the diameter of $C$, denote it by $C(\delta)$.
%	%	Also, for any $\delta' < \delta$, we still call it the diameter of $C$.
%\end{rem}

\section{Proof of Theorem \ref{thm general} for degree $1$ integer-valued generalized polynomials}\label{linear case}

In this section, we will prove Theorem \ref{thm general} for degree $1$ integer-valued generalized polynomials. 

\medskip

The proof of the next lemma  is identical with the proof of  Lemma \ref{lemma:spectra of IP-set is IP 1}, and we so omit it.

\begin{lem}\label{lemma:spectra of IP-set is IP}
	Let $A$ be an IP-set and $p\in \widetilde{SGP_1}$ with $A(p(n)) \neq 0$.
	Then 
	$$\{p(n) \in \mathbb{Z}: n\in A \}$$ 
	is also an IP-set. 
\end{lem}

%\begin{proof}
%%	We may assume that 
%%	$p(n)=\sum\limits_{i=1}^{\rm{t_1}} \left\lceil b_i\left\lceil \alpha_i n\right\rceil \right\rceil-\sum\limits_{j=1}^{\rm{t_2}} \left\lceil c_j\left\lceil \beta_j n\right\rceil \right\rceil$ with $t_1,t_2\in \N, \a_i, \b_i\in \R, i=1, \dots, t_1$ and $\beta_j, c_j\in \mathbb{R}, j=1, \dots, t_2.$ Moreover, 
%%	$$ A(p(n))=\sum_{i=1}^{t_1}b_i \alpha_i- \sum_{j=1}^{t_2}c_j \beta_j \neq 0.$$
%\end{proof}

\begin{thm}\label{Main theorem: degree1+d=1}
	Let $(X, T)$ be a topologically mildly mixing minimal system and $p\in \widetilde{SGP_1}$ with $A(p(n)) \neq 0$.
	Then for any non-empty open subsets $U, V$ of $X$,
	$$N(p, U, V)= \{ n\in \mathbb{Z}: U\cap T^{-p(n)}V \neq \emptyset \}$$
	is an IP$^*$-set.
\end{thm}

\begin{proof}
	We may assume that $$p(n)=\sum\limits_{i=1}^{\rm{t_1}} \left\lceil b_i\left\lceil \alpha_i n\right\rceil \right\rceil-\sum\limits_{j=1}^{\rm{t_2}} \left\lceil c_j\left\lceil \beta_j n\right\rceil \right\rceil
	, n\in \mathbb{Z}, $$ 
	with $t_1, t_2\in \mathbb{N}$, $\alpha_i, b_i\in \mathbb{R},$ $\beta_j, c_j\in \mathbb{R},$ $i=1, \dots, t_1$,
	 $j=1, \dots, t_2$ and $\sum_{i=1}^{t_1}b_i\alpha_i-\sum_{j=1}^{t_2}c_j\beta_j \neq 0.$
	
	Since $(X, T)$ is a mildly mixing minimal system, one has $A= \{n\in \mathbb{Z}: U\cap T^{-n}V \neq \emptyset\} $ is an IP$^*$-set.
	If $N(p, U, V)$ is not an IP$^*$-set, 
	there exists an IP-set $F$ such that $F\cap N(p, U, V) =\emptyset$.
	That is for any $m\in F$, 
	$$U\cap T^{-p(m)}V =\emptyset.$$
	By Lemma \ref{lemma:spectra of IP-set is IP},
	one has $\{p(m)\in \Z: m\in F\}$ is an IP-set. 
	then $$\{ p(m)\in \Z: m\in F\} \cap A\neq \emptyset.$$
	Choose $m\in F$ such that $p(m)\in A$, that is $U\cap T^{-p(m)}V \neq \emptyset ,$ so there is a contradiction,
	thus $N(p, U, V)$ is an IP$^*$-set.
	 
\end{proof}

\medskip
First we prove an even more special case.

\begin{thm} \label{thm special}
	Let $(X, T)$ be a topologically mildly mixing minimal system, $d\in \mathbb{N}$, and $p_1, \dots, p_d\in \widetilde{SGP_1}$ with $(p_1,\dots,p_d)$ non-degenerate. Then for all non-empty open subsets $U , V_1, \dots, V_d $ of $X$,
	$$\{n\in \Z: U\cap T^{-p_1(n) }V_1\cap \dots \cap T^{-p_d(n) }V_d \neq \emptyset \}$$
	is an IP$^*$-set.
\end{thm}
\begin{proof}	
	We will prove by induction on $d$ that for all non-empty open sets $U, V_1,\dots, V_d$ of $X$, any IP-set $A$ there exists $n\in A$ such that 
\begin{equation*}
		U\cap T^{-p_1(n) }V_1\cap \dots \cap T^{-p_d(n) }V_d \neq \emptyset.
\end{equation*} 
	
If $d=1$, by Lemma \ref{Main theorem: degree1+d=1}, $N(p_1, U,V_1)$ is an IP$^*$-set. Now we assume that the result holds for $d > 1$. Next we will show that the result holds for $d+1$.  Let $U, V_1, \dots, V_d, V_{d+1}$ be non-empty open subsets of $X$. We will show that there exists $n\in A$ such that 
\begin{equation*}
U\cap T^{-p_1(n) }V_1\cap  \dots \cap T^{-p_{d+1}(n) }V_{d+1} \neq \emptyset.
\end{equation*}

By Lemma \ref{lem-goodm}, there exists $\delta>0,$ $Q\subset \widehat{SGP_f}$ (for some $f\in \mathbb{N}$) and $$C=C(\delta,Q)=\bigcap_{q(n)\in Q}\{n\in \mathbb{Z}:\{q(n)\}\in (-\delta,\delta)\}$$ such that for each $m\in C$, $m$ is good w.r.t. $ p_i(n),i=1,\dots, d+1.$ Let $\tilde{C}=C\cap A$. Since $C$ is an IP$^*$-set, one has $\tilde{C}$ is an IP-set (by Lemma \ref{IP+IPstar}). 
We may assume that $\tilde{C}=\{n_{\alpha}: \alpha\in \mathcal{F}\}$ for some sequence $\{n_i\}_{i=1}^{\infty} \subset \mathbb{Z}$.
	
Since $(X, T)$ is minimal, there is some $l\in \mathbb{N}$ such that $X=\cup_{j=0}^{l}T^{j}U$. 
By Lemma \ref{Main lemma:descending sequence}, 
there are non-empty open subsets $V_1^{(l)}, \dots, V_{d+1}^{(l)}$ of $X$ and $\alpha_0<\alpha_1 <\dots <\alpha_l\in \mathcal{F}$ such that for each $i=1, \dots, d+1$, 
one has $$T^{p_i(n_{\alpha_j})}T^{-j}V_{i}^{(l)} \subset V_i,  \ \text{for all } \ 0\le j \le l.$$
Notice that $n_{\alpha_j} \in \tilde{C} \subset C,j=0,1,\dots, l$ is good w.r.t. $ p_i(n),i=1, \dots, d+1.$  By Lemma \ref{lem-proper}, there is a Nil$_1$ Bohr$_0$-set $C_1$ associated to $\{p_1, \dots, p_{d+1}\}$ and $\{ n_{\alpha_0}, \dots, n_{\alpha_l}\}.$ Thus we have for any $i\in \{1,\dots, d+1\}$ and each $j\in \{0,\dots, l\},$ $p_i(n+n_{\alpha_j})=p_i(n)+p_i(n_{\alpha_j})$ when $n\in C_1.$
	
Let $q_i=p_{i}-p_1 \in \widetilde{SGP_1}$, $i=2, \dots, d+1$. Now applying the induction hypothesis, there is some
%$\{n_{\beta}: \beta>\alpha_l\}$ is also an IP-set
 $\beta\in \mathcal{F}$ such that $\beta>\alpha_l$, $n_{\beta}\in C_1\cap \{n_{\alpha}: \alpha\in \mathcal{F}\}$ and $$ V_1^{(l)} \cap T^{-q_2(n_{\b})}V_2^{(l)} \cap \dots \cap T^{-q_{d+1}(n_{\b})}V_{d+1}^{(l)}\neq \emptyset.  $$ Hence there is some $x\in V_1^{(l)}$ such that $T^{p_{i}(n_{\beta})-p_1(n_{\beta})}x\in V_i^{(l)}$ for $i=2, \dots, d+1.$ Clearly, there exists $y\in X$ such that $T^{p_1(n_{\beta})}y=x$. Since $X=\cup_{j=0}^{l}T^jU$, there is some $b\in \{0, 1, \dots, l\}$ such that $T^{b}z=y$ for some $z\in U$. Thus for each $i=1, \dots, d+1$,
\begin{align*}
     T^{ p_i (n_{\beta})+p_i (n_{\a_b})}z
	&=T^{ p_i (n_{\beta})} T^{ p_i (n_{\a_b}) }T^{-b}T^{-p_{1}(n_{\beta})}x \\
	&=T^{ p_i(n_{\a_b})  }T^{-b}T^{p_{i}(n_{\beta})-p_1(n_{\beta})}x\\
	%&= T^{p_i(n_{\a_b}) }T^{-b}T^{q_{i}(n_{\beta})}x 	\\
	&\subseteq T^{p_i (n_{\a_b}) }T^{-b} V^{(l)}_i  \subseteq V_i.	
\end{align*}
Since $n_{\beta}\in C_1,$ one has
$$p_i(n_{\beta}+n_{\a_b})= p_i (n_{\beta})+ p_i(n_{\a_b})$$
for all $i=1,\dots,d+1$. 
Hence we have, 	$$z \in U\cap T^{-p_1(n_{\a})} V_1 \cap \dots \cap T^{-p_d (n_{\a})}V_d  \cap T^{- p_{d+1}(n_{\a})} V_{d+1},$$ where $\alpha=\beta \cup \a_{b}$ as $\b \cap \a_{b}=\emptyset$ and $n_{\a}\in A.$ By induction the proof is complete.
\end{proof}

Now we can prove our main result for degree $1$ integer-valued generalized polynomials.
\begin{thm}\label{thm simple: degree 1}
	Let $(X, T)$ be a topologically mildly mixing minimal system, $d\in \mathbb{N}$ and $p_1, \dots, p_d\in \widetilde{GP_1}$ with $(p_1,\dots,p_d)$ non-degenerate. Then for all non-empty open subsets $U , V_1, \dots, V_d $ of $X$, $$\{n\in \Z: U\cap T^{-p_1(n) }V_1\cap \dots \cap T^{-p_d(n) }V_d \neq \emptyset \}$$	is an IP$^*$-set.
\end{thm}
\begin{proof}
	Let $p_1, \dots, p_d\in \widetilde{GP_1}$. By Lemma \ref{lem sgp_gp}, there exist $h_i(n) \in \widetilde{SGP_1}$, $i=1,\dots,d$ and  $C=C(\delta,q_1,\dots,q_k)$ such that 
	$$p_i(n)=h_i(n), \forall n \in C,  i=1,\dots,d,$$
	where $\delta$ is small enough and $q_1,\dots,q_k\in \widehat{SGP_1}$.
	
	Set $$N=\{n \in \mathbb{Z}: U\cap T^{-h_1(n)}V_1 \cap \dots \cap T^{-h_d(n)}V_d \neq \emptyset \},$$
	by Theorem \ref{thm simple: degree 1}, $N$ is an IP$^*$-set, and thus
	$N\cap C$ is also an IP$^*$-set. Since for any $n\in N\cap C \subseteq C$, $p_i(n)=h_i(n),i=1,\dots,d$, we have
	$$n\in \{n \in \mathbb{Z}: U\cap T^{-p_1(n)}V_1 \cap \dots \cap T^{-p_d(n)}V_d \neq \emptyset \},$$ that is, \[N\cap C\subseteq \{n \in \mathbb{Z}: U\cap T^{-p_1(n)}V_1 \cap \dots \cap T^{-p_d(n)}V_d \neq \emptyset \},\]
	hence $$\{n\in \Z: U\cap T^{-p_1(n) }V_1\cap \dots \cap T^{-p_d(n) }V_d \neq \emptyset \}$$	is an IP$^*$-set.
\end{proof}

\section{The PET-induction and the generalized polynomial extension of van der Waerden's theorem} \label{PET}
In this section, we will use the PET-induction (Polynomial Exhaustion Technique) to prove the generalized polynomial form of van der Waerden's theorem.
The PET-induction was introduced by Bergelson in \cite{Bergelson87}.
See also \cite{BL96,BL99,Leibman94} for more on the PET-induction. We firstly review some notions from \cite{ZZ2021} (most of which is similar to that used in \cite{Bergelson87}).

\medskip

\begin{de}
Let $p(n), q(n) \in \widetilde{SGP}$, we denote $p \sim q$ if 
$\deg(p)=\deg(q)$ and $\deg(p-q)<\deg(p)$.
One can easily check that $\sim$ is an equivalence relation. A \emph{system} $P$ is a finite subset of $\widetilde{SGP}$.
Given a system $P$, we define its \emph{weight vector} $\Phi(P)=(\omega_1, \omega_2, \dots ) $,
where $\omega_i$ is the number of equivalent classes under $\sim$ of degree $i$ special integer-valued generalized polynomials represented in $P$.
For distinct weight vectors
$ \Phi(P)= (\omega_1, \omega_2, \dots)$ and $\Phi(P')= (\upsilon_1, \upsilon_2, \dots)$,
one writes $\Phi(P)>\Phi(P')$ if $\omega_d>\upsilon_d$, where $d$ is the largest $j$ satisfying
$ \omega_j \neq \upsilon_j$,
then we say that $P'$ \emph{precedes} $P$.
This is a well-ordering of the set of all possible weight vectors and the PET-induction is simply an induction scheme on this ordering.
\end{de}

For example, let $P=\{3n^2- \left\lceil \sqrt{5}n\right\rceil, \left\lceil \pi n^3 \left\lceil \sqrt{2}n\right\rceil\right\rceil + \left\lceil \frac{1}{5}n^3  \right\rceil, n \left\lceil \pi n \right\rceil , \left\lceil \pi n^3 \left\lceil \sqrt{2}n\right\rceil\right\rceil + \left\lceil \pi n^2  \right\rceil \}$,
then $\Phi(P)=(0, 2, 0, 1, 0, \dots )$.

\medskip

Now we are ready to show the following generalized polynomial extension of van der Waerden's theorem, which follows the  train of thought established in the proof of the polynomial  van der Waerden's theorem in \cite{Bergelson06}.

\begin{thm}\label{generalized BL96}
	Let $(X ,T)$ be a dynamical system with the metric $\rho$ and $p_1, \dots, p_d\in \widetilde{SGP}.$ 
	Then for any $\ep>0$, there exists $x\in X$ and $n\in \N$ such that
	\begin{equation}\label{a11}
	\rho(T^{p_t(n)}x,x)<\ep
	\end{equation}
	for all $t=1,\dots ,d$ simultaneously. Moreover, the set
	$$\{n\in \Z : \forall \ep >0, \exists x\in X \ \text{s.t.} \ \forall t\in \{1, \dots ,d\}, \ \eqref{a11}\ \text{is satisfied} \}$$
	is an IP$^*$-set.
\end{thm}

\begin{proof}
%	Passing if need to a suitable subset of $X$, we may assume that the system $(X,T)$ is minimal.
	Without loss of generality, we may assume that $(X,T)$ is minimal.
	
	First of all notice that Theorem \ref{generalized BL96} holds trivially if all  $p_t(n)$, $t=1,\dots , d$ are zeros. 
	We shall prove (using PET-induction) that Theorem \ref{generalized BL96} is valid for any system by showing that its validity for an arbitary system $\SS =\{p_1, \dots, p_d\}$ follows from its validity for all the systems preceding $\SS$.
	
	Given $\ep >0$ and let $A=FS(\{s_i\}_{i\in\N})$ be an IP-set. For $n\in A,$ $n=s_{i_1}+\dots +s_{i_m},$ define its support $\sigma(n)$ by $\sigma(n)=\{i_1,\dots,i_m\}$ and let 
	\[A_n=FS(\{s_i\}_{i\in \N\setminus\sigma(n)});\]
	for $n_1,\dots,n_q\in A$ define
	$$A_{n_1,\dots,n_q}=FS(\{s_i\}_{i\in \N\setminus \cup^q_{j=1}\sigma(n_j)}).$$
	It is clear from the definition for IP-set that for any $n\in A$ and $l\in A_n,$ we have $n+l\in A.$

	By Lemma \ref{lem-goodm}, there exists $\delta>0,$ $Q\subset \widehat{SGP_a}$ (for some $a\in \mathbb{N}$) and 
	$$C=C(\delta,Q)=\bigcap_{q(n)\in Q}\{n\in \mathbb{Z}:\{q(n)\}\in (-\delta,\delta)\}$$ 
	such that for each $m\in C$, $m$ is good w.r.t. $p_t(n)\ \text{for}\ 1\le t \le d.$ Let $\tilde{A}=A\cap C$, then $\tilde{A}$ is also an IP-set. 
	We may assume that  $\tilde{A}=FS(\{\tilde{s}_i\}_{i\in\N})$.
	Moreover, denote 
	$$\tilde{A}_{n_1,\dots,n_q}=FS(\{\tilde{s}_i\}_{i\in \N\setminus \cup^q_{j=1}\sigma(n_j)})\subseteq \tilde{A}$$ 
	with $n_1,\dots,n_q\in \tilde{A}.$ Notice that $\tilde{A}_{n_1,\dots,n_q}$ are still IP-sets for any $n_1,\dots,n_q\in \tilde{A}.$

	Let $p_1$ be of the minimal degree in the system $\SS$, 
	we may assume that $\SS$ does not contain trivial polynomials, and so $\deg(p_1)\geq1.$  
	Consider the system 
	$$\SS_0 = \{p_2 -p_1,\ p_3 -p_1, \dots, p_d-p_1\}$$
	(if $d=1,$ $\SS_0$ is empty).
	
	Notice that the elements of $\SS$ nonequivalent to $p_1 $ do not change their degree and the equivalence of one to another after they have been subtracted by $p_1$; on the other hand, the degrees of the elements of $\SS$ which are equivalent to $p_1$ do decrease after these elements have been subtracted by $p_1$. 
	Hence, the number of equivalence classes with the minimal degree in $\SS$ decreases by $1$ when we pass from $\SS$ to $\SS_0$ (although some new equivalence classes with smaller degrees can arise in $\SS_0$). This means that $\Phi(\SS_0)<\Phi(\SS)$, and by the PET-induction hypothesis, the statement of the theorem is valid for $\SS_0.$  
	
	Therefore, one can choose $y_0\in X,$ $n_1\in \tilde{A}$ such that
	\[\rho(T^{p_t(n)-p_1(n)}y_0, y_0)<\frac{\ep}{2}\ \text{for}\ t=2\ ,\dots ,d. \]  
	%(if $d=1,$ let $y_0\in X$ and $n_1\in \tilde{A}$ be arbitrary). 
	Denote $x_0=y_0,$ $x_1=T^{-p_1(n_1)}y_0$, then
	\[T^{p_1(n_1)}x_1=x_0,\ \rho(T^{p_t(n_1)}x_1,x_0)<\frac{\ep}{2}\ \text{for}\ t=2,\dots, d.\]     
	We will find a sequence of points $x_0, x_1, x_2, \dots\in X$ and a sequence of integers $n_1,n_2,\dots$ with $n_1\in \tilde{A},$ $n_{m+1}\in \tilde{A}_{n_1,\dots, n_m},$ $m\in \N$ such that for every $l,m,l<m,$ one has 
	\begin{equation}\label{ml}
	\rho(T^{p_t(n_m+\dots +n_{l+1})}x_m,x_l)<\frac{\ep}{2}\ \text{for any}\ t=1,\dots, d.
	\end{equation}
	The points $x_0,x_1$ and the integer $n_1$ have already been chosen, suppose that $x_m\in X,\ n_m\in \tilde{A}_{n_1,\dots, n_{m-1}}$ have been chosen. 
	The inequality $\eqref{ml}$ holds not only for $x_m$ but also for all the points of the $\ep_m$-neighborhood of $x_m$ for some $\ep_m,$ $0<\ep_m<\frac{\ep}{2}.$ 
	Since $(X,T)$ is assumed to be minimal, there is $N\in \mathbb{N}$ such that $X=\cup_{i=0}^N T^{-i}B(x_m, \frac{\ep_m}{2})$,
	where $B(x_m, \frac{\ep_m}{2})=\{x\in X: \rho(x, x_m)<\frac{\ep_m}{2}\}$.
   %for every $y\in Y$ there exists $b=b(y)\in \{0,1,2\dots,N\}$ such that $d(T^{b}y, x_m)<\ep_m.$ 
   Choose $\delta_m>0$ such that for any $y, y'\in X$ with $\rho(y, y')<\delta_m$,
   one has 
   $\rho(T^by, T^by')<\frac{\ep_m}{2}$, for all $b=0, \dots, N$.
%	Choose $\delta_m$ such that, for every $y\in X$ there is some $b,$ $1\leq b \leq N,$ so that the inequality $d(y,y')<\delta_m$ implies 
Thus for every $y\in X$, there is some $b,$ $0\leq b \leq N$ such that the inequality $\rho(y,y')<\delta_m$ implies 
 	\begin{equation}\label{by T continuity}
 	\rho(T^b y', x_m)<\ep_m.
 	\end{equation}
	Let 
	\begin{align*}
		&\{p_{t,l}(n)\}_{1\leq t \leq d,\ 0\leq l \leq m} \\
		=&\begin{cases}
			p_t(n+n_m+\dots+n_{l+1})-p_t(n_m+\dots+n_{l+1})-p_1(n) , &\text{if} \ 0\leq l \leq m-1, \\   p_t(n)-p_1(n) , &\text{if} \ l=m. \end{cases}  	
	\end{align*}
	(Notice that if $l=m-1,$ then $n_m+\dots+n_{l+1}=n_m.$)

	Since $n_1\in \tilde{A},n_2\in \tilde{A}_{n_1},\dots,n_{m}\in \tilde{A}_{n_1,\dots, n_{m-1}},$  
	one gets 
	$$\{n_m,n_m+n_{m-1},\dots,n_m+\dots+n_1\}\subset \tilde{A}.$$ 
	By Lemma \ref{lem-proper}, there is a Nil$_{h_1}$ Bohr$_0$-set $C_1$  associated to $\{p_1, \dots, p_d\}$ and $\{n_m,n_m+n_{m-1},\dots,n_m+\dots+n_1\}$. For any $1\leq t \leq d,\ 0\leq l \leq m-1$ and $n\in C_1$,
	put $$D(p_t(n),n_m+\dots+n_{l+1})=p_t(n+n_m+\dots+n_{l+1})-p_t(n_m+\dots+n_{l+1})-p_t(n), $$
	then $$D(p_t(n),n_m+\dots+n_{l+1}) \in \widetilde{SGP}$$ 
	with $\deg(D(p_t(n), n_m+\dots+n_{l+1}))<\deg(p_t)$. 
%	and for all $n\in C_1$, 
%	$$D(p_t(n),n_m+\dots+n_{l+1}) =p_t(n+n_m+\dots+n_{l+1})-p_t(n_m+\dots+n_{l+1})-p_t(n).$$
	
	Let $q_{t,l}(n)=D(p_t(n),n_m+\dots+n_{l+1})+p_t(n)-p_1(n)$ if $0\leq l \leq m-1$ and $q_{t,m}(n)=p_t(n)-p_1(n)$. Then for all $0\le l \le m$, $q_{t,l}(n) \in \widetilde{SGP}$ and $$p_{t,l}(n)=q_{t,l}(n), \forall n\in C_1.$$
		Denote
	\[\SS_m=\{q_{t,l}:1\le t\le d,\  0\leq l \leq m\}.\]
	If $p_t(n)$ is not equivalent to $p_1(n)$, the elements $q_{t,0}(n),\dots,q_{t,m}(n)\in \SS_m$  are equivalent to   $p_t(n)$  and their equivalence is preseved, that is, if $p_t(n)$ is equivalent to $p_s(n)$ then $q_{t,j}(n)$ is equivalent to $q_{s,i}(n)$ for every $j,i=0,\dots ,m$. If $p_t(n)$ is equivalent to $p_1(n)$, then the degrees of these elements decrease: $\deg(q_{t,j})<\deg(p_t)=\deg(p_1(n)),\ 0\leq j \leq m$. 
	So the number of equivalence classes having degrees greater than $\deg(p_1(n))$ does not change, whereas the number of equivalence classes of elements having the minimal degree in $\SS_m$ decreases by 1 when we pass from $\SS$ to $\SS_m$.
	Then $\SS_m$ precedes $\SS$ and by our PET-induction hypothesis, the conclusion of Theorem \ref{generalized BL96} holds for the system $\SS_m.$
	
	Hence, we can find $y_m\in X,$ $n_{m+1}\in \tilde{A}_{n_1,\dots, n_m}\cap C_1$ such that 
	$\rho(T^{q_{t,l}(n_{m+1})}y_m, y_m)<\delta_m$ for every $t$, $1\leq t \leq d$, and any $l$, $0\leq l\leq m.$ 
	In light of $p_{t,l}(n)=q_{t,l}(n), \forall n\in C_1,$ one has $\rho(T^{p_{t,l}(n_{m+1})}y_m, y_m) <\delta_m$,  $1\leq t \leq d$, $0\leq l\leq m.$ Choose $b \in \{0, 1, \dots, N\}$ such that $\eqref{by T continuity}$ holds for all $y'$ from the $\delta_m$-neighborhood of $y_m $, denote $x_{m+1}=T^{-p_1(n_{m+1})}T^{b}y_m.$ Then, for each $t=1,\dots ,d,$ we have 
	\begin{align*}
		&\rho(T^{p_t(n_{m+1}+n_m+\dots+n_{l+1})}T^{-p_t(n_m+\dots +n_{l+1})}x_{m+1},x_m)\\
		=&\rho(T^bT^{p_{t,l}(n_{m+1})}y_m,x_m) \\
		<&\ep_m, \quad l=0,\dots, m-1,
%	&\rho(T^bT^{p_{t,l}(n_{m+1})}y_m,x_m) \\
%	=&\rho(T^{p_t(n_m+\dots +n_{l+1})}T^{p_t(n_{m+1}+n_m+\dots+n_{l+1})}x_{m+1},x_m) \\
%	<&\ep_m, \quad l=0,\dots, m-1,
	\end{align*}
	and
	\begin{align*}
	&\rho(T^{p_t(n_{m+1})}x_{m+1},x_m)\\
	=&\rho(T^{p_t(n_{m+1})}T^{-p_1(n_{m+1})}T^{b}y_m,x_m)\\
	<&\ep _m.
	\end{align*}
	Hence by the choice of $\ep_m,$ %i.e.,by inequality \eqref{ml}
	one gets 
	\[\rho(T^{p_t(n_{m+1}+n_m+\dots +n_{l+1})}x_{m+1},x_l)<\frac{\ep}{2}, \ \text{for}\ l=0,\dots, m-1\]
	and
	\[\rho(T^{p_t(n_{m+1})}x_{m+1},x_m)<\frac{\ep}{2}.\]
	By the above procedure, we have found the sequence of points $x_0,x_1,x_2,\dots\in X$ and the sequence of integers $n_1,n_2,\dots$ with $n_1\in \tilde{A},$ $n_{m+1}\in \tilde{A}_{n_1,\dots, n_m},$ $m\in \N$ satisfying the inequality \eqref{ml}.
	Since $X$ is compact, there exist $l,m,l<m,$ such that $\rho(x_l, x_m)<\frac{\ep}{2}.$ 
	Put $n=n_m+\dots+n_{l+1}\in \tilde{A}\subseteq A,$ then $\rho(T^{p_t(n)}x_m,x_l)<\frac{\ep}{2},$ for each $t,$ $1\leq t\leq d,$ and hence $\rho(T^{p_t(n)}x_m, x_m)<\ep.$
\end{proof}

\begin{cor}\label{generalized BL96 minimal version}
	If $(X,T)$ is minimal and $p_1, \dots, p_d\in \widetilde{SGP},$ then for any IP-set $A$, any non-empty open set $V\subset X$, there exists $n\in A$ such that
	$$V\cap T^{-p_1(n)} V\cap \dots \cap T^{-p_d(n)} V\neq \emptyset.$$

\end{cor}

\medskip

\section{Proof of Theorem \ref{thm general}}

In this section, we  proceed to use the PET-induction to prove our main result Theorem \ref{thm general}. 
As in Section \ref{linear case}, we first prove the special case (Theorem \ref{thm simple: degree 1}). Likewise, before carrying out our main result we focus on the elements in $\widetilde{SGP}$. We start with a brief overview of our idea of this proof by the PET-induction, precisely, in order to prove the Theorem \ref{thm general} for any system $P$ in $\widetilde{SGP}$, we will start from $\Phi(P)=(d,0,\dots)$ (this is true by Theorem \ref{thm special}). After that, we assume the result holds for any system $P'$ with $\Phi(P')<\Phi(P)$. Then we show the result holds for $P$, and we complete the proof.

\medskip
 
To simplify our exposition, we introduce some notations  given by Zhang and Zhao in \cite{ZZ2021}. 

\begin{de}
	Let $a, b\in \mathbb{R}$, $p\in  \widetilde{SGP}$ and $N_0=1000.$ We denote  $a>>_N b$ if $a>b>0$ and $a>Nb$; we denote $a>>_{N_0} b$ by $a>>b$; we use the symbol $a \approx_N  b$ when $|a|>>_N|a-b|$ and $|b|>>_N|a-b|$; by $a \approx b$ we mean $a \approx_{N_0}b$; $p =_{C} q$ iff $p(n)=q(n)$ for any $n \in C$.
\end{de}
 
%\begin{itemize}
%	\item Let $a>b>0$, we denote $a>>b$ iff there exists a large enough $N>0$  such that $a>N(b+1)$.
%	\item $a \approx b$ iff $|a|>>|a-b|$ and $|b|>>|a-b|$.	
%	\item $p(n)=_{C}q(n)$ iff $p(n)=q(n)$ for any $n \in C$.
%\end{itemize}	

%\begin{lem}\cite{ZZ2021}\label{lem-large-approx}
%	Let $a, b, a_1, \dots, a_k\in \mathbb{R}, k\in \mathbb{N}$.
%	\begin{enumerate} 
		
%		\item If $|a|>> 1$ then $\left\lceil a \right \rceil \approx a$.
		
%		\item If $|a|>>1$, $|b|>>1$, $a\approx a'$ and $b \approx b'$, then $ab \approx a'b'$. 
%		 Moreover, if it still satisfies $|a+b|>>1$, then $a'+b'\approx a+b$.
		
%		\item If  $|\sum_{i=1}^{k}a_i|>>1$, and for any $1 \le i \le k$, $|a_i|>>1$, $a_i\approx a_i'$, then    $|\sum_{i=1}^ka_i'|>>1$.
		
%	\end{enumerate}
%\end{lem}

\medskip

%For instance,
%$10000\sqrt{2}>>1$, $5000\sqrt{3}>>1$, $|10000\sqrt{2}-5000\sqrt{3}|>>1$, 
%$\left \lceil 10000\sqrt{2} \right\rceil \approx 10000\sqrt{2}, 
%\left \lceil 5000\sqrt{3} \right\rceil \approx 5000\sqrt{3},$
%$10000\sqrt{2} \times 5000\sqrt{3} \approx \left \lceil 10000\sqrt{2} \right\rceil \left \lceil 5000\sqrt{3} \right\rceil.$

Recall that $A(p(n))$ is the sum of the coefficients of the maximal-degree components of the generalized polynomial $p(n)$ (see Definition \ref{def-nondegenerate}).  We have the following lemmas from \cite{ZZ2021}.
 
\begin{lem}\cite[Lemma 4.3]{ZZ2021} \label{lem-A(p)-general}
	Let $h(n) =\sum_{k=1}^l c_k\left \lceil p_k(n) \right \rceil \in \widetilde{SGP_d}$, where $c_k \in \mathbb{Z}$, $p_{k}(n) =\prod_{i=1}^{t}L(a_{i,1}n^{j_{i,1}},\dots, a_{i,l_{i}}n^{j_{i,l_i}})$ with $|a_{i,m}|>>1$, $1\leq i\leq t, 1\leq m\leq l_i$, $|A(h(n))|>>1$ and $\deg(h)\ge 2$. Let $m(h)=\frac{2\sum_{k=1}^{l}|c_k \deg(p_k(n))A(p_k(n))|}{|\sum_{k=1}^{l}c_k\deg(p_k(n))A(p_k(n))|}.$ There are  Nil Bohr$_0$-sets $C, $ $C_1$ such that for any $m\in C,$ $ |m|>m(h) $  we can find $D(h(n),m) \in \widetilde{SGP_{d-1}}$ with $\deg(D(h(n),m)) <\deg(h(n))$ such that 
	
	$$D(h(n),m)=_{C_1}h(n+m)-h(n)-h(m),$$ 
	$$A(D(h(n),m))\approx_{2N_0} \deg(h) mA(h(n)).$$
		We call $D(h(n),m)$ the derivative of $h(n)$ w.r.t. $m$.
	
\end{lem}

\begin{lem}\cite[Lemma 4.5]{ZZ2021}\label{lem-multi-q(ij)}
	Let $d\in \mathbb{N}$, $p_i \in \widetilde{SGP}, i=1, \dots, d$, with $(p_1,\dots,p_d)$ non-degenerate, $\deg(p_i)\ge 2, 1\le i \le d$ and each $p_i$ satisifies conditions in Lemma \ref{lem-A(p)-general} and $|A(p_i-p_j)|>>1,1\leq i\neq j\leq d.$ Then there exists a Nil Bohr$_0$-set $C$ and a sequence $\{r(n)\}_{n=0}^{\infty}$ of natural numbers, such that for any $l \in \mathbb{N}$ and $k_0,k_1,\dots,k_l \in C$	with $|k_0|>r(0)$ and $|k_{i}|>|k_{i-1}|+r(|k_{i-1}|),1\leq i\leq l$,
	there exists a Nil Bohr$_0$-set $C_1$ and $q_{i,j}\in \widetilde{SGP}$ with
	$$q_{i,j}(n):=D(p_i(n),k_j)+p_i(n)-p_1(n)=_{C_1}p_i(n+k_j)-p_i(k_j)-p_1(n)$$
	and
	\begin{equation}\label{eq-a(qij-q)}
	|A(q_{i,j}(n))|>>1, |A(q_{i,j}(n)-q_{i',j'}(n))|>>1 
	\end{equation}
	for all $ (i,j) \neq (i',j')\in \{1,\dots,d\}\times\{0,1,\dots,l\}$.
\end{lem}

\medskip

%If $S$ is a set and $k\in \N,$ let us call a subset $B\in S$ a $k$-subset if $|B|=k.$
%
%\begin{thm}\cite[Theorem A]{Ramsey1929}\label{infinitary Ramsey theory}
%	Let $S$ be a countable infinite set and let $k,r\in \N.$ If the $k$-subsets of $S$ are $r$-colored then there exists an infinite subset $T\subseteq S$ having the property that the $k$-subsets of $T$ comprise a monochromatic family.
%\end{thm}

%The following lemma shows that we can assume that for any non-degenerate system $(p_1, \dots, p_d)$, $p_1, \dots, p_d\in \widetilde{SGP}$, the coefficients satisfy that $$

\begin{lem}\label{lem-g(n)-g(kn)}
	Let $(X, T)$ be a dynamical system, $U, V_1, \dots, V_d$ be non-empty open subsets of $X$ and $q \in \Z$.  
	We denote
	$$N=\{n\in \mathbb{Z}: U\cap T^{-p_1(n)}V_1 \cap \dots \cap T^{-p_d(n)} V_d \neq \emptyset\},$$
	$$N_1=\{n\in \mathbb{Z}: U\cap T^{-p_1(qn)}V_1 \cap \dots \cap T^{-p_d(qn)} V_d \neq \emptyset\}.$$
	If  $N_1$ is an IP$^*$-set, then $N$ is an IP$^*$-set.
\end{lem}

\begin{proof}
	We notice that for any $n\in N_1,$ $qn\in N.$ Thus $qN_1\subseteq N.$ Since $N_1$  is an IP$^*$-set and by Lemma \ref{qA}, one gets $N$ is an IP$^*$-set.	
\end{proof}

\begin{rem}
	Notice that for any $0\neq a \in \mathbb{R}$ and $k\in \mathbb{N}$, if we choose $$C=\{n \in \mathbb{Z}: \{abn^k\}\in (-\frac{1}{4}, \frac{1}{4}), \{a \lel bn^k \rr \} \in (-\frac{1}{4}, \frac{1}{4}), \{bn^k\}\in (-\frac{1}{4|a|},-\frac{1}{4|a|})\},$$ 
	then we have $\left\lceil a \left \lceil bn^k\right \rceil\right\rceil=_C \left \lceil abn^k\right\rceil$. 
	Combining this fact with Lemma \ref{lem-g(n)-g(kn)}, from now on we always assume that for any
	$p_1, \dots, p_d\in \widetilde{SGP}$ with $(p_1, \dots, p_d)$ non-degenerate, $|A(p_i)|>>1$ and $|A(p_i-p_j)|>>1$ for any $1 \le i\neq j \le d$.
%	and all $p_i(n),1\le i\le d$ satisfy the conditions in Lemma \ref{lem-A(p)-general}.
\end{rem}

\medskip

Before showing the proof of Theorem \ref{thm general}, we show the case of special integer-valued generalized polynomials.

\begin{thm} \label{thm general sgp}
Let $(X, T)$ be a topologically mildly mixing minimal system, $p_1, \dots, p_d\in \widetilde{SGP}$ and $(p_1, \dots,p_d)$ be non-degenerate, $d\in \mathbb{N}$. Then $(X, T)$ is $\{p_1, \dots,p_d\}_{\D}$-IP$^*$-transitive. That is, for all non-empty open subsets $U , V_1, \dots, V_d $ of $X$, $$\{n\in \Z: U\cap T^{-p_1(n) }V_1 \cap \dots \cap T^{-p_d(n) }V_d \neq \emptyset \}$$ is an IP$^*$-set.
\end{thm}

\begin{proof}
	 
	Let $d\in \mathbb{N}$ and $P=\{p_1, \dots, p_d\}$. 
	We start from the system whose weight vector is $(d, 0, \dots)$. That is, the degree of every  element of $P$ is $1$. This case is Theorem \ref{thm special}.
	
	Now let $P \subset \widetilde{SGP}$ be a system whose weight vector $> (d, 0, \dots)$, and we assume that for all systems $P'$ preceding $P$, we have $(X,T)$ is  $P'_{\D}$-IP$^*$-transitive. Now we show that $(X,T)$ is  $P_{\D}$-IP$^*$-transitive. 
%	More precisely, in Claim 1 we will show that $*_1$ holds for $P'$ and $*_2$ hold for $P'$ imply that $*_1$ holds for $P$, in Claim 2 we will show that $*_1$ holds for $P$ and $*_2$ holds for $P'$ imply that $*_2$ holds for $P$.
	\medskip
	
	{\noindent\bf Step 1.} \label{ main claim 1}
	$(X, T)$ is $P$-IP$^*$-transitive.
	\medskip
	
	%{\noindent\bf Proof of Claim 1:}
	Since the intersection of two IP$^*$-sets is still an IP$^*$-set, it is sufficient to show that for any $p\in P$, and for any given non-empty open subsets $U, V$ of $X$,
	$$N(p,U,V)=\{n\in \mathbb{Z}: U\cap T^{-p(n)}V \neq \emptyset\}$$
	is an IP$^*$-set.
	
%	It suffices to show that for any IP-set $A=\{n_{\alpha}\}_{\alpha\in \mathcal{F}}$, for all $\alpha_0\in \mathcal{F}$, there is $\alpha>  \alpha_0$ such that $$U\cap T^{-p{(n_{\alpha})}}V\neq \emptyset.$$
	
	If $\deg(p)=1$, by Lemma \ref{Main theorem: degree1+d=1}, $N(p,U,V)$ is an IP$^*$-set.
	
	Now we assume that $\deg(p) \ge 2$.	Given an IP-set $A.$ By Lemma \ref{lem-goodm}, there exists $\delta>0,$ $Q\subset \widehat{SGP_b}$ (for some $b\in \mathbb{N}$) and $$C=C(\delta,Q)=\bigcap_{q(n)\in Q}\{n\in \mathbb{Z}:\{q(n)\}\in (-\delta,\delta)\}$$ such that for each $m\in C$, $m$ is good w.r.t. $ p(n).$ 
	Let $\tilde{C}=C\cap A$. Since $C$ is an IP$^*$-set, one has $\tilde{C}$ is an IP-set. 
    We may assume that $ \tilde{C} =\{n_{\alpha}:\alpha\in \mathcal{F}\}$ for some sequence $\{n_i\}_{i=1}^{\infty}\subset \mathbb{Z}$. 
    As $(X, T)$ is minimal, there is some $N\in \mathbb{N}$ such that $X=\cup_{i=0}^{N}T^{i}U$.
	%	Let $h(m, n)=p(n+m)-p(n)-p(m)$.
	For each $i=0, 1, \dots, N$, by Corollary \ref{generalized BL96 minimal version}, 
	there are some $y_i\in T^{i}V$ and $\alpha_i\in \mathcal{F}$ such that $|n_{\a_i}|>m(p)$ (see Lemma \ref{lem-A(p)-general} for the notation $m(p)$) and 
	$$T^{p(n_{\alpha_i})}y_i \in T^{i}V, \ i=0, 1, \dots, N.$$
	We may assume that $\alpha_0<\alpha_1 <\dots <\alpha_N$ and $n_{\a_i}\neq n_{\a_j}$ where $0\leq i\neq j\leq N.$
	Let $V_i$ be a neighborhood of $y_i$ such that 
	$$T^{p(n_{\alpha_i})}V_i \subseteq T^{i}V, \ i=0, 1, \dots, N.$$
	%	Let 
	%	$$\mathcal{S}'=$$
	By Lemma \ref{lem-A(p)-general}, 	there is a Nil Bohr$_0$-set $C_1$ and $D(p(n),n_{\alpha_i}) \in \widetilde{SGP_{d-1}}$ with $$\deg(D(p(n),n_{\alpha_i})) <\deg(p(n)), i=0, \dots, N$$ such that $$D(p(n),n_{\alpha_i})=_{C_1} p(n_{\alpha_i}+n)-p(n_{\alpha_i})-p(n).$$ 
    Let $$P'=\{D(p(n),n_{\alpha_i}): i=0, 1, \dots, N\}.$$
	Then $P' \subset \widetilde{SGP}$. 
	Since for any $i=0, \dots, N$, $\deg(D(p(n),n_{\alpha_i}))<\deg(p(n))$, thus we have $\Phi(P') < \Phi(\{p\})$. 
	
	For any $i\neq i'\in  \{ {0,1, \dots ,N} \} $, since $n_{\a_i}\neq n_{\a_{i'}}$,  by Lemma \ref{lem-A(p)-general},
	$$A(D(p(n),n_{\alpha_i}))\approx_{2N_0}\deg(p)n_{\alpha_i}A(p(n)), \ |A(D(p(n),n_{\alpha_i}))|>>1,$$ 	
	$$A(D(p(n),n_{\alpha_{i'}}))\approx_{2N_0}\deg(p)n_{\alpha_{i}'}A(p(n)),\  |A(D(p(n),n_{\alpha_{i'}}))|>>1,$$
	$$A(D(p(n),n_{\alpha_i})-D(p(n),n_{\alpha_{i'}}))\approx\deg(p)(n_{\alpha_i}-n_{\alpha_{i}'})A(p(n)),$$ then $$|A(D(p(n),n_{\alpha_i})-D(p(n),n_{\alpha_{i'}}))|>>1.$$
	
	By the inductive hypothesis, there are some $x\in X$ and $\beta \in \mathcal{F}$ with $\beta >\alpha_N,$ and $n_{\beta}\in C_1\cap \{n_{\alpha}: \alpha\in \mathcal{F}\}$ such that
	$$T^{D(p(n_{\b}),n_{\alpha_i})} x\in V_i, \ \forall i\in \{0, 1, \dots, N\}.$$
	Then we have 
	\[T^{ p(n_{\alpha_i}+n_{\b})-p(n_{\alpha_i})-p(n_{\b})}x\in V_i, \ \forall i\in \{0, 1, \dots, N\}.\]
	$$T^{p(n_{\alpha_i}+n_{\beta})-p(n_{\beta})}x \in T^{ p(n_{\alpha_i})}V_i \subseteq T^{i}V,  \ \forall i\in \{0, 1, \dots, N\}.$$
	Hence 
	$$T^{-i}T^{-p(n_{\beta})}x \in T^{ -p( n_{\alpha_i}+ n_{\beta} ) }V, \ \forall i\in \{0, 1, \dots, N\}.$$
	Since $X=\cup_{i=0}^{N} T^{i}U$, there is some $i_0\in \{0, 1, \dots, N\}$ such that $T^{-i_0}T^{-p(n_{\beta})}x \in U$,
	and thus 
	$$U\cap T^{-p(n_{\alpha})}V \neq \emptyset,$$
	where $\alpha=\alpha_{i_0}\cup \beta$ as $\alpha_{i_0} \cap \beta =\emptyset$.
	That is,
	$$N(p,U, V) \cap A\neq \emptyset.$$
	Since $A$ is an arbitrary IP-set, 
	which shows that $(X,T)$ is $P$-IP$^*$-transitive.

	\bigskip
	
	\medskip
	
	{\noindent\bf Step 2.}	
	$(X,T)$ is  $P_{\D}$-IP$^*$-transitive. 
	
	\medskip
	
	We will show that for all non-empty open subsets $U, V_1,\dots, V_d$ of $X$, any IP-set $A$, there exists $n\in A$ such that 
	\begin{equation*}
	U\cap T^{-p_1(n) }V_1\cap \dots \cap T^{-p_d(n) }V_d \neq \emptyset.
	\end{equation*} 
	%{\noindent\bf Proof of Claim 2:}
	By permuting the indices, we may assume that $\deg(p_1)\leq \deg(p_2)\leq\dots \leq \deg(p_d).$ Presume that $\deg(p_w)=1$ and $\deg(p_{w+1})\ge 2$, $1\le w<d$. 
	If for any $p\in P$, $\deg p\ge 2$, we put $w=0$.
	Let $\{r(n)\}_{n=0}^{\infty}$ be the sequence in Lemma \ref{lem-multi-q(ij)} w.r.t. $(p_{w+1},\dots,p_d)$.

	Since $(X, T)$ is minimal, there is some $N\in \mathbb{N}$ such that $X=\cup_{i=0}^{N}T^iU$. By Lemma \ref{lem-goodm}, there exists $\delta>0,$ $Q\subset \widehat{SGP_f}$ (for some $f\in \mathbb{N}$) and $$C=C(\delta,Q)=\bigcap_{q(n)\in Q}\{n\in \mathbb{Z}:\{q(n)\}\in (-\delta,\delta)\}$$ such that for each $m\in C$, $m$ is good w.r.t. $ p(n).$ Let $\tilde{C}=C\cap A$. Since $C$ is an IP$^*$-set, one has $\tilde{C}$ is an IP-set. Clearly, there exists an IP-set $\{n_{\alpha}: \alpha\in \mathcal{F}\}\subseteq \tilde{C}\subseteq A.$ By Lemma \ref{Main lemma:descending sequence}, there are $\{\alpha_j\}_{j=0}^N \subset \mathcal{F}$ (where $\a_0<\a_1<\dots <\a_N$) such that $|n_{\alpha_0}|> r(0), |n_{\alpha_j}|>|n_{\alpha_{j-1}}|+r(|n_{\alpha_{j-1}}|)$  for all $j=1, \dots, N$,
	and $V_t^{(N)} \subset V_t$ for $t=1, \dots, d$ such that
	\begin{equation}
	T^{p_t (n_{\alpha_j})}T^{-j} V_t^{(N)} \subset V_t, \ j=0, \dots, N.		
	\end{equation}

	%	By Claim 1, $(X,T)$ is P-thickly-syndetic transitive.
	%	Then by Lemma \ref{lem-deg-1},
	%	there are integers $\{k_j\}_{j=0}^{l} \subset \tilde{C}$
	%	and non-empty open sets $V_i^{(l)}\subset V_i, 1\leq i\leq d$
	%	such that
	%	$|k_j| >|k_{j-1}|+r(|k_{j-1}|)$ for $j=0, \dots, l \ (k_{-1}=0)$
	%	and
	%	$$ T^{p_i(k_j)}T^{-j}V_i^{(l)} \subset V_i, 0\leq j\leq l, 1\leq i\leq d. $$
	
	By Lemma \ref{lem-proper}, there is a Nil Bohr$_0$-set $C_1$  associated to $\{p_1, \dots, p_d\}$ and  $\{n_{\alpha_0},  \dots, n_{\alpha_N}\}$. 
%	Put $C_1=C_1'\cap C_1''$, then $C_1\in \mathcal{F}_{h, 0}$, where $h=\max \{h_1, h_2\}$.
%	Without loss of generality, we may assume that $\frac{\varepsilon}{2}$ is as in Lemma \ref{asso for special cases}.	
	This means that for any $(t, j)\in \{1, \dots, d\} \times \{0, \dots, N\}$, 
	there exists $D(p_t(n), n_{\alpha_j}) \in \widetilde{SGP}$ with $\deg(D(p_t(n), n_{\alpha_j}))<\deg(p_t(n))$ such that
	$$D(p_t(n),n_{\alpha_j}) =_{C_1}p_t(n_{\alpha_j}+n)-p_t(n_{\alpha_j})-p_t(n).$$
	For $1\le t\le d$ and $0\le j\le N$, let $$p_{t, j}(n)=p_t(n_{\alpha_j}+n)-p_t(n_{\alpha_j})-p_1(n)$$ and
	$$q_{t, j}(n)=D(p_t(n),n_{\alpha_j})+p_t(n)-p_1(n),$$ then $q_{t,j}(n) \in \widetilde{SGP}$ and $$p_{t,j}(n)=_{C_1}q_{t,j}(n).$$

	For $w\ge 1, 1 \le t \le w$, since $\deg(p_t)=1$, we have $D(p_t(n),n_{\alpha_j})=0$,  
	and then $$q_{t,j}(n)=p_t(n)-p_1(n),\ j=0,1,\dots,N.$$ 
	Let $$P'=\{p_2(n)-p_1(n),\dots,p_w(n)-p_1(n)\}\cup \{q_{t,j}(n):t=w+1, \dots, d, j=0, 1, \dots, N\}.$$ 
	Then $P'\subset \widetilde{SGP}$ and when $t\in \{w+1,\dots, d\},$   $q_{t,0}(n),\dots,q_{t,N}(n)$ are equivalent to $p_t(n).$ So, the number of equivalence classes having degrees greater than $1$ does not change, whereas the number of equivalence classes of polynomials having the minimal degree $1$ in $P$ decreases by 1 when we pass from $P$ to $P'$, hence $\Phi(P')<\Phi(P).$  Notice that when $2\leq s\neq t \leq w$,
	$$|A(q_{s, i}-q_{t, j})|=A(|p_s-p_t|)>>1,  i,j\in\{0,1,\dots,N\};$$ 
	when $2\leq s \leq w$ and $w+1\leq t \leq d$, 
	then 
	$$q_{s, i}-q_{t, j}=p_s(n)-D(p_t(n),n_{\alpha_j})-p_t(n), \ i,j\in\{0,1,\dots,N\},$$  
	thus  
	$$|A(q_{s, i}-q_{t, j})|=|A(p_{t})|>>1, \ i,j\in\{0,1,\dots,N\};$$
	 when $2\leq s, t \leq w,$ 
	 since  $$|n_{\alpha_0}|> r(0), |n_{\alpha_j}|>|n_{\alpha_{j-1}}|+r(|n_{\alpha_{j-1}}|), j=1, \dots, N$$ 
	  by Lemma \ref{lem-multi-q(ij)}, one has 
	  $$|A(q_{s, i})|>>1, |A(q_{s,i}-q_{t,j})|>>1, i,j \in\{0,1,\dots,N\}.$$
	
	For $w=0$, if $p_t(n)$ is not equivalent to $p_1(n)$, $q_{t, j}\sim p_t,$ for $ (t,j)\in \{1,\dots, d\}\times\{0,1,\dots,N\}$. 
	If $p_t(n)$ is equivalent to $p_1(n)$, then the degrees of these elements decrease: $\deg(q_{t,j})<\deg(p_t)=\deg(p_1),\ 0\leq j \leq N$. In this case $$P'= \{q_{t,j}(n): (t,j)\in \{1,\dots, d\}\times\{0,1,\dots,N\}\} .$$ 
	We still have $P'\subset \widetilde{SGP}$ and $\Phi(P')<\Phi(P)$. 
	Since $|n_{\alpha_0}|> r(0), |n_{\alpha_j}|>|n_{\alpha_{j-1}}|+r(|n_{\alpha_{j-1}}|)$  for all $j=1, \dots, N$, by Lemma \ref{lem-multi-q(ij)}, one has $$|A(q_{t, j})|>>1\ \text{and}\ |A(q_{t,j}-q_{s,i})|>>1$$ for all $ (t,j) \neq (s,i)\in \{1,2,\dots,d\}\times\{0,1,\dots,N\}.$
	
	Both of the above conditions satisfy  the inductive hypothesis and therefore there is  $x\in V_1^{(N)}$ and $\beta\in \mathcal{F}$ such that $\beta>\alpha_N$, $n_{\beta}\in C_1\cap \{n_{\alpha}: \alpha\in \mathcal{F}\}$  and 
	$$T^{p_{t, j}(n_{\beta})}x=T^{q_{t,j}(n_{\beta})} x \in V_t^{(N)},  1\leq t\leq d, 0\le j \le N.$$
	Let $y=T^{-p_1(n_{\beta})}x$. 
	Then by $X=\cup_{i=0}^{N}T^iU$ there is $z\in U$ and $0\le b \le N$ such that $y=T^bz$.
	Then $z=T^{-b}T^{ -p_1(n_{\beta}) } x$ and we have for each $1\le t \le d$,
	\begin{align*}
	T^{p_t(n_{\beta} + n_{\alpha_b} )}z &= T^{ p_t (n_{\beta} + n_{\alpha_b}) } T^{-b} T^{-p_1( n_{\beta} )}x \\
	&= T^{p_t (n_{\alpha_b})} T^{-b} T^{p_t(n_{\beta}+n_{\alpha_b}) - p_t(n_{\alpha_b})-p_1(n_{\beta} )} x \\
	&= T^{p_t (n_{\alpha_b})}T^{-b} T^{p_{t, b} (n_{\beta})} x \\
	& \in T^{p_t(n_{\alpha_b})}T^{-b}V_t^{(N)} \subseteq V_t. 
	\end{align*}
	This implies that
	$$z\in U\cap T^{-p_1(n_{\alpha})}V_1 \cap \dots \cap T^{-p_d(n_{\alpha})}V_d$$
	with $\alpha=\beta \cup \alpha_b \in \mathcal{F}$ as $\beta \cap \alpha_b =\emptyset$. Since $n_{\alpha}\in A$, it follows that $(X,T)$ is  $P_{\D}$-IP$^*$-transitive. Thence the proof of the theorem is complete.
	
\end{proof}

\medskip

We are ready to prove Theorem \ref{thm general}.

\begin{proof}[Proof of Theorem \ref{thm general}]
	
	Let $p_1, \dots, p_d \in \mathcal{G}$. Then by Lemma \ref{lem sgp_gp}, there exist $h_i(n) \in \widetilde{SGP}$, $i=1,2,\dots,d$ and  $C=C(\delta,q_1,\dots,q_t)$ such that $p_i(n)=h_i(n), \forall n \in C,  i=1,2,\dots,d$.
	
	Set $$N=\{n \in \mathbb{Z}: U\cap T^{-h_1(n)}V_1 \cap \dots \cap T^{-h_d(n)}V_d \neq \emptyset \},$$
	by Theorem \ref{thm general sgp}, $N\cap C$ is also an IP$^*$-set. Since for any $n\in N\cap C \subseteq C$, $p_i(n)=h_i(n),i=1,2,\dots,d$, we have
	$$n\in \{n \in \mathbb{Z}: U\cap T^{-p_1(n)}V_1 \cap \dots \cap T^{-p_d(n)}V_d \neq \emptyset \},$$ that is, \[N\cap C\subseteq \{n \in \mathbb{Z}: U\cap T^{-p_1(n)}V_1 \cap \dots \cap T^{-p_d(n)}V_d \neq \emptyset \},\]
	hence $$\{n\in \Z: U\cap T^{-p_1(n) }V_1\cap T^{-p_2(n)}V_2\cap \dots \cap T^{-p_d(n) }V_d \neq \emptyset \}$$	is an IP$^*$-set.
\end{proof}

\noindent {\bf Acknowledgments}

\bigskip
Y. Cao is supported by China Postdoctoral Science Foundation (Grant no.~2023M741635). J. Zhao (corresponding author) is supported by NSF of China (Grant no.~12301226) and NSF of Zhejiang Province (Grant no.~LQ23A010006).
%The authors would like to thank Professor X. Ye  for  help discussions.
%The first author were supported by NNSF of China(11871188,11671094), the second author were supported by NNSF of China (11431012).

%%%%%%%%%%%%%%%%%%%%%%%%%%%%%%%%%%%%%%%%%%%%%%%%%%%%%%%%%%%%%%%%%%%%%%%%%%%%%%%%%%%%%%%%%%%
%%%%%%%%%%%%%%%%%%%%%%%%%%%%%%%%%%%%%%%%%%%%%%%%%%%%%%%%%%%%%%%%%%%%%%%%%%%%%%%%%%%%%%%%%%%
\bibliographystyle{plain}
\bibliography{refcy}

\end{document}